\documentclass[letterpaper, 11pt,  reqno]{amsart}

\usepackage{amsmath,amssymb,amscd,amsthm,amsxtra, esint, bbm, enumerate, relsize}
\usepackage{mathtools} % For floor and celing functions
\usepackage{skak, bbm} % chess pieces
\usepackage{ulem}
\usepackage[implicit=true]{hyperref}

\usepackage{mparhack}

\allowdisplaybreaks

\usepackage{color}

\definecolor{gr}{rgb}   {0.,   0.69,   0.23 }
\definecolor{bl}{rgb}   {0.,   0.5,   1. }
\definecolor{mg}{rgb}   {0.85,  0.,    0.85}
\definecolor{yl}{rgb}   {0.8,  0.7,   0.}
\makeatletter
\@namedef{subjclassname@2020}{%
  \textup{2020} Mathematics Subject Classification}
\makeatother

\newtheorem{theorem}{Theorem} [section]

\newtheorem{lemma}[theorem]{Lemma}
\newtheorem{proposition}[theorem]{Proposition}
\newtheorem{remark}[theorem]{Remark}

\newtheorem{corollary}[theorem]{Corollary}

\newtheorem*{ackno}{Acknowledgment}

\DeclareMathOperator*{\supp}{supp}

\newcommand{\noi}{\noindent}
\newcommand{\Z}{\mathbb{Z}}
\newcommand{\R}{\mathbb{R}}

\newcommand{\T}{\mathbb{T}}

\let\Re=\undefined\DeclareMathOperator*{\Re}{Re}
\let\Im=\undefined\DeclareMathOperator*{\Im}{Im}

\newcommand{\PP}{\mathbb{P}}

\newcommand{\E}{\mathbb{E}}

\renewcommand{\L}{\mathcal{L}}

\newcommand{\NB}{\mathbb{N}}

\newcommand{\FL}{\mathcal{F}L} %%%%%%%%%%%% Fourier-Lebesgue spaces

\newcommand{\m}{\mathfrak{m}}

\newcommand{\al}{\alpha}

\newcommand{\dl}{\delta}

\newcommand{\eps}{\varepsilon}
\newcommand{\kk}{\kappa}
\newcommand{\g}{\gamma}

\newcommand{\s}{\sigma}

\newcommand{\ft}{\widehat}

\newcommand{\wt}{\widetilde}
\newcommand{\cj}{\overline}
\newcommand{\dx}{\partial_x}
\newcommand{\dt}{\partial_t}
\newcommand{\dd}{\partial}

\newcommand{\ta}{\theta}

\newcommand{\tr}{\text{tr}}

\newcommand{\Dr}{\Theta}

\newcommand{\A}{\mathcal{A}}

\DeclarePairedDelimiter\ceil{\lceil}{\rceil}
\DeclarePairedDelimiter\floor{\lfloor}{\rfloor}

\renewcommand{\l}{\ell}
\renewcommand{\o}{\omega}
\renewcommand{\O}{\Omega}

\newcommand{\les}{\lesssim}

\newcommand{\vk}{\varkappa}

\newcommand{\jb}[1]
{\langle #1 \rangle}

\newcommand{\ind}{\mathbbm{1}}

\DeclareMathOperator{\Id}{Id}
\DeclareMathOperator{\sgn}{sgn}

\numberwithin{equation}{section}
\numberwithin{theorem}{section}

\usepackage{marginnote}

\begin{document}
\baselineskip = 12.7pt

\title[Invariant measures for KdV and mKdV]
{A continuum of invariant measures for the periodic KdV and mKdV equations
}
\author[A.~Chapouto, J.~Forlano]
{Andreia Chapouto and Justin Forlano }

\address{
Andreia Chapouto \\
Laboratoire de math\'ematiques
de Versailles\\
 UVSQ, Universite Paris-Saclay\\
CNRS, 45 avenue des \'Etats-Unis, \\
78035 Versailles\\
Cedex, France\\ \and\\
School of Mathematics\\
The University of Edinburgh\\
and The Maxwell Institute for the Mathematical Sciences\\
James Clerk Maxwell Building\\
The King's Buildings\\
Peter Guthrie Tait Road\\
Edinburgh\\
EH9 3FD\\
 United Kingdom\\
 \and\\
Department of Mathematics\\
University of California\\
Los Angeles\\
CA 90095\\
USA}

\email{andreia.chapouto@uvsq.fr}

\address{
Justin Forlano\\
School of Mathematics\\
9 Rainforest Walk\\
Monash University VIC 3800, Australia\\ \and \\
School of Mathematics\\
The University of Edinburgh\\
and The Maxwell Institute for the Mathematical Sciences\\
James Clerk Maxwell Building\\
The King's Buildings\\
Peter Guthrie Tait Road\\
Edinburgh\\
EH9 3FD\\
 United Kingdom\\
 \and\\
Department of Mathematics\\
University of California\\
Los Angeles\\
CA 90095\\
USA}

\email{justin.forlano@monash.edu}

\subjclass[2020]{35Q53, 35Q55, 60H30}

\keywords{Korteweg-de Vries equation, modified Korteweg-de Vries equation, nonlinear Schr\"{o}dinger equation, invariant measures, Gibbs measures, complete integrability.}

\begin{abstract}
	We consider the real-valued defocusing modified Korteweg-de Vries equation (mKdV) on the circle.
Based on the complete integrability of mKdV, Killip-Vi\c{s}an-Zhang (2018) discovered a conserved quantity
	which they used to prove low regularity a priori bounds for solutions.
It has been an open question if this conserved quantity can be used to define invariant measures supported at fractional Sobolev regularities.
	Motivated by this question, we construct probability measures supported on $H^s(\T)$ for $0<s<1/2$ invariant under the mKdV flow.
	We then use the Miura transform to obtain invariant measures for the Korteweg-de Vries equation, whose supports are rougher than the white noise measure.
	We also obtain analogous results for the defocusing cubic nonlinear Schr\"{o}dinger equation. These invariant measures cover the lowest possible regularities for which the flows of these equations are well-posed.

\end{abstract}

\maketitle

\vspace*{-6mm}

\section{Introduction}

We consider the real-valued modified Korteweg de-Vries equation (mKdV):
\begin{equation}\label{mkdv}
	\dt q=-\dx^{3} q \pm 6   q^2  \dx q,
\end{equation}
posed on the circle $\T=\R /\Z$. This equation \eqref{mkdv} is also known as the Miura mKdV equation after \cite{Miura}.
We say that \eqref{mkdv} is defocusing with the positive sign in front of the nonlinearity, and focusing with the negative sign.
The mKdV equation \eqref{mkdv} has garnered much attention from the mathematical community due to its rich structure, as it is both Hamiltonian and completely integrable.
The Hamiltonian for mKdV is given by
\begin{align}
	H^{\text{mKdV}}(q) = \frac{1}{2}\int_{0}^{1} (\dx q(x))^2 \pm  q(x)^4 \, dx, \label{Hmkdv}
\end{align}
which generates the dynamics of \eqref{mkdv} through the Poisson bracket
\begin{align}
	\{ F,G\} = \int_{0}^{1} \frac{\dl F}{\dl q} \dx \frac{\dl G}{\dl q}dx.
	\label{Poisson}
\end{align}
A natural question, motivated by the finite-dimensional Hamiltonian setting, is the existence of probability measures which are invariant under the flow of an infinite-dimensional Hamiltonian system. This has been a very active field of research in the past thirty years, initiated by Lebowitz-Rose-Speer~\cite{LRS} and Bourgain~\cite{BO94, BO96}; see also \cite{Friedlander, Zhidkov91, McKean95}. In this paper, we aim to construct new invariant measures for mKdV, inspired by the recent work of Tzvetkov \cite{T23} and the conserved quantities discovered by Harrop-Griffiths, Killip, Vi\c{s}an, and Zhang~\cite{KVZ, HGKV}. We also extend this construction to the Korteweg-de Vries equation (KdV) via the Miura map.

A natural candidate for an invariant measure is the Gibbs measure
\begin{align}
	``d\rho_1 = Z^{-1}\exp( - H^{\text{mKdV}}(q)-M(q))dq". \label{gibbs}
\end{align}
Here, $Z$ is a normalisation constant and $M$ is the mass
\begin{align*}
	M(q) =\frac{1}{2}\int_{0}^{1} q(x)^2 \, dx,
\end{align*} which is also conserved under \eqref{mkdv}. We have included it in \eqref{gibbs} to better align with our following discussions.
The representation in \eqref{gibbs} is purely formal, but one can rigorously define the measure $\rho_1$ as a weighted measure of the form
\begin{align}
	d\rho_{1} =Z^{-1} \exp\big( \mp  \tfrac{1}{2} \textstyle\int_{0}^1 q(x)^4 dx \big)d\wt{\mu}_{1}, \label{Gibbs2}
\end{align}
where $\wt{\mu}_{1}$ is the Gaussian measure on $L^2(\T)$ with covariance operator $(-\dd_{x}^{2}+1)$. In the defocusing case, there is a ``good" sign: the density in \eqref{Gibbs2} is integrable with respect to $\wt{\mu}_{1}$ since $\int_{0}^{1} q^4 dx$ is non-negative and almost surely finite in the support of $\wt{\mu}_1$ (by Sobolev embedding). In the focusing case, there is a ``bad" sign which motivated Lebowitz-Rose-Speer~\cite{LRS} to insert a cut-off depending on the conserved mass:
\begin{align}
	d\rho_{1,R} =Z^{-1} \ind_{\{ \|q\|_{L^2}^{2}\leq R\}} \exp\big(   \tfrac{1}{2} \textstyle\int_{0}^1 q(x)^4 dx \big)d\wt{\mu}_{1}. \label{Gibbs2R}
\end{align}
This allowed them to rigorously construct the measures $\rho_{1,R}$ for any $R>0$.
The invariance of the Gibbs measures \eqref{Gibbs2} and \eqref{Gibbs2R} under the flow of \eqref{mkdv} was proven by Bourgain \cite{BO94}.

In order to discuss the invariance of measures such as $\rho_1$, we first need dynamics, at least defined almost surely on their support. Much is known on the well-posedness of mKdV \eqref{mkdv}; see \cite{Kato1, MTsut, BO2, BO94, CKSTT1, TT, NTT, MPV2, KapTop, Molinet, Schippa, Forlano} and references therein.
In particular, exploiting complete integrability, Kappeler-Topalov~\cite{KapTop} proved global well-posedness in $L^2(\T)$. An alternative proof was given by the second author in~\cite{Forlano}, using the method of commuting flows in~\cite{KV, HGKV}.
These results are sharp as Molinet~\cite{Molinet} proved ill-posedness below $L^2(\T)$, as well as global existence without uniqueness in $L^2(\T)$; see also \cite{Schippa}.

Returning to the problem of invariant measures, the complete integrability of mKdV~\eqref{mkdv}
implies the existence of an infinite number of conserved quantities $\mathcal{E}_{k}(q)$, for $k\in\NB$, which control the $H^k (\T)$-norm of solutions. See for example \cite{Lax} in the context of the also completely integrable KdV
\begin{align}
	\dt q = -\dx^3 q + 6q \dx q. \label{KdV}
\end{align}
Consequently, the probability measures
\begin{align}
	``d\rho_{k} =Z^{-1}_{k} \exp(- \mathcal{E}_{k}(q) -M(q)) dq" \label{rhok}
\end{align}
are also natural candidate invariant measures for \eqref{mkdv}, which would be supported on ever more regular spaces as $k\in \NB$ increases.
This program of constructing an infinite family of invariant measures associated to the explicit higher-order conservation laws was completed for other completely integrable equations, namely for KdV and the cubic nonlinear Schr\"{o}dinger equation (NLS) by Zhidkov \cite{Zhidkov, Zhid}, for the Benjamin-Ono equation (BO) by Tzvetkov, Visciglia, and Deng in a series of papers \cite{T10, TV1, TV2, DTV, Deng}, and also recently for the intermediate long wave equation \cite{LOZ, CLOZ}. Such a construction is expected to also hold for mKdV.

One of the main analytical reasons that invariant measures and associated dynamics are of interest is that one can use the invariance of the measure as a tool to upgrade local well-posedness on the support of the measure to global well-posedness. This argument is known as Bourgain's invariant measure argument after \cite{BO94, BO96}, and is useful and interesting when no deterministic global well-posedness is known or exists. In contrast, the aforementioned program of constructing invariant measures associated to higher conservation laws does not benefit from the globalization aspect of Bourgain's invariant measure argument as typically deterministic global well-posedness on the support of the measures is already well-known! Nonetheless, the presence of an invariant measure and associated dynamics allows us to access viewpoints from the theory of dynamical systems. Namely, the solutions are recurrent in the sense of the Poincar\'e recurrence theorem; see \cite[p. 106]{ZhidPoin} and Corollary~\ref{COR:Poincare} below.

The question that motivates our paper is if there exist intermediate invariant measures \textit{between} $\rho_k$ and $\rho_{k+1}$ in \eqref{rhok}. This was very recently shown to be the case for BO by Tzvetkov \cite{Tz}. Namely, he constructed non-degenerate invariant measures for BO supported on $H^{\s}(\T)$ for any $\s>-\frac12$,
thus covering every regularity regime in which BO is globally well-posed. His proof uses the Birkhoff coordinates for BO in \cite{GKT}.
Here, we aim to take the first step towards answering this question for mKdV by constructing, for each $0<\s<\frac12$, a (canonical) invariant measure supported on $H^{\s}(\T)$; namely, a continuum of invariant measures. Our approach does not use the existence of Birkhoff coordinates.
Instead, we exploit the completely integrable nature of the equation from the perspective of the commuting flows method in \cite{KV, HGKV} and the ideas in \cite{KVZ}.
 In particular, our argument combines tools from PDE theory, completely integrable systems, and stochastic analysis.

We now motivate our construction of the invariant measures.
In \cite{KVZ}, Killip-Vi\c{s}an-Zhang observed that for mKdV (and also for NLS, with appropriate modifications), the quantity
\begin{align}
	\al(\kk,q) =- \sum_{m=1}^{\infty} \frac{(\mp 1)^{m-1}}{m} \tr\{  [ (\kk-\dd)^{-1}q(\kk+\dd)^{-1}q]^{m}\},  \quad \kk \geq 100(1+\|q\|_{L^2(\T)}^{2}) \label{alpha}
\end{align}
is conserved, at least for regular enough solutions. Here, $\al(\kk,q)$ is a series expansion of the logarithm of the perturbation determinant; see \cite{KVZ}.
They then exploited this conservation to prove a-priori $H^s$ bounds for every $-\frac12<s<0$.\footnote{Strictly speaking, the authors did not use \eqref{alpha} but a very similar version of it, which is necessary for potentials below $L^2(\T)$.}
Unfortunately, we believe that $\al(\kk,q)$ and conserved quantities based on it are unsuitable choices to define candidate invariant measures of the form \eqref{gibbs}.
The main obstructions are the sign indefiniteness of $\al(\kk,q)$ and the restriction on $\kk$ in \eqref{alpha}. Both suggest that a mass cut-off is necessary when defining the measure. However, this would force $\kk$ to depend upon the size of the mass cut-off, and so would $\al(\kk,q)$, albeit in a highly nonlinear and non-explicit way.
We believe such a restriction to be artificial.

Fortunately, in \cite{HGKV}, a more suitable candidate was discovered. They showed that $\al(\kk,q)$ can be written in terms of a macroscopic conservation law $A(\kk,q)$, which extends by analyticity to all $\kk\geq 1$. Following \cite{KMV, FKV},
we thus consider the quantity $A(\kk,q)$, whose precise definition we defer to \eqref{defA}.
Note that $A(\kk,q)$ is well-defined and conserved under the defocusing mKdV flow for any $\kk\geq 1$ and $q\in C^{\infty}(\T)$; see \cite{FKV}. In the following, we will thus restrict to the defocusing mKdV equation; see Remark~\ref{RMK:focusing} for a discussion on the focusing case.

To recover an $H^s(\T)$-like quantity from $A(\kk,q)$, we follow an idea in \cite{KVZ}. Given $\kk\geq1$ and $\frac{1}{2}<s<1$, we consider \begin{align}\label{Acal:intro}
	\mathcal{A}(\kk,q) = A(\kk,q)-\tfrac{1}{2}A(\tfrac{\kk}{2},q),
\end{align}
and the following candidate for a conserved quantity
\begin{align}\label{Wkk}
	\mathcal{E}_{s}(q) : = \int_1^\infty \kk^{2s} \mathcal{A}(\kk,q) \, d\kk.
\end{align}
It follows from the conservation of $A(\kk, q)$ under the mKdV flow that $\A(\kk, q)$ is also conserved, and thus we expect the same for $\mathcal{E}_{s}(q)$.
We isolate the quadratic-in-$q$ contributions in $\mathcal{E}_s(q)$ by writing
\begin{align}
	\mathcal{E}_{s}(q) = \tfrac{1}{2} \| \mathfrak{m}_{s}(-i\dd) q\|_{L^{2}(\T)}^{2} + \int_1^\infty \kk^{2s} \, \mathcal{V}(\kk,q) \, d\kk, \label{EsV}
\end{align}
where $\mathfrak{m}_s(-i\partial)$ has Fourier multiplier
\begin{align}
	\mathfrak{m}_{s}(\xi) &: =   \bigg( \int_1^\infty  \kk^{2s-1} w(\xi,\kk) \, d\kk\bigg)^{\frac{1}{2}}, \qquad \text{where } \quad
	w(\xi,\kk) :=  \frac{3\kk^2 \xi^2}{(\kk^2 + \xi^2) (4\kk^2 + \xi^2)}, \label{wmult}
\end{align}
and $\mathcal{V}(\kk,q)$ contains the contributions in $\mathcal{A}(\kk, q)$ which are at least quartic in $q$; see \eqref{EEdefn}.

Since $w(0,\kk) =0$, the weight in \eqref{wmult} vanishes at the zero-th frequency. To avoid a restriction to mean-zero functions, we consider Gaussian measures $\mu_{s}$ defined by
\begin{align}
	d\mu_{s}:= Z_{s}^{-1}\exp\big( -\tfrac{1}{2}  \| \mathfrak{m}_{s}(-i\dd) q\|_{L^{2}(\T)}^{2} -\tfrac{1}{2}\|q\|_{L^2(\T)}^{2}\big) dq . \label{gauss0}
\end{align}
Rigorously, we view $\mu_{s}$ as the pushforward measure under the map
\begin{align}
	\o\in \O \mapsto q(x;\o) = \sum_{\xi\in 2\pi\Z} \frac{g_{\xi}(\o)}{ \jb{\m_{s}(\xi)}}e^{ i\xi x}, \label{randomfourier}
\end{align}
where $\jb{\xi}:=(1+\xi^2)^{\frac{1}{2}}$ and $(g_{\xi}(\o))_{\xi\in 2\pi\Z}$ are standard complex-valued normal random variables living on some ambient probability space $(\O, \mathcal{F}, \mathbb{P})$, with the restrictions that $g_{-\xi}=\cj{g_{\xi}}$ and $g_0$ is real.
Since $\jb{\mathfrak{m}_s(\xi)} \sim \jb{\xi}^s$ (see Lemma~\ref{LEM:ms}), it follows that $$q(\cdot; \o) \in H^\s(\T) \setminus H^{s-\frac12}(\T)$$ almost surely for any $\s < s- \frac12$.

We can finally define our candidate measures $\rho_{s,R}$ given by
\begin{align}
	d \rho_{s, R} := Z_{s,  R}^{-1}\, \ind_{\{\|q\|_{L^2}^2 \leq R\}} \exp\bigg(- \int_1^\infty\kk^{2s} \mathcal{V}(\kk, q)  \, d\kk \bigg) d\mu_{s}, \label{GibbsEmeas}
\end{align}
for $\frac12<s<1$ and $R>0$.
Our first main result is a rigorous construction of \eqref{GibbsEmeas} and corresponding invariance under the flow of the defocusing mKdV equation \eqref{mkdv}.

\begin{theorem}[Invariant measures for defocusing mKdV]\label{THM:mkdv}
	Let $\tfrac{1}{2}<s<1$ and $R>0$. Then, $\rho_{s,R}$ in \eqref{GibbsEmeas} defines a probability measure on $L^2(\T)$, endowed with the Borel sigma algebra, which satisfies: \\
	\textup{(i)} $\rho_{s,R}$ is absolutely continuous with respect to the Gaussian measure $\mu_{s}$ in \eqref{gauss0}.\\
	\textup{(ii)} $\supp \rho_{s,R} \subseteq H^{\s}(\T)\setminus H^{s-\frac{1}{2}}(\T)$, for every $\s<s-\tfrac{1}{2}$, and moreover,
	\begin{align}
		\bigcup_{R>0} \supp \rho_{s, R} = \supp \mu_{s}. \label{supppsr}
	\end{align}
	\textup{(iii)}  The measure $\rho_{s,R}$ is invariant under the defocusing mKdV flow \eqref{mkdv}. More precisely, for $\Phi^{\textup{mKdV}}(t):q^0\mapsto q(t)$ the data-to-solution map for defocusing mKdV, we have
	\begin{align}
		\int_{L^2(\T)} f( \Phi^{\textup{mKdV}}(t)(q)) d\rho_{s,R}(q) =\int_{L^2(\T)} f( q) d\rho_{s,R}(q)
		\label{mkdvinv}
	\end{align}
	for all $t\in \R$ and for all $f\in L^1(L^2(\T); d\mu_{s})$.
\end{theorem}

In view of Theorem~\ref{THM:mkdv} (ii), our measures are supported on rougher spaces as compared to the corresponding Gibbs measure \eqref{Gibbs2}. In this sense, Theorem~\ref{THM:mkdv} constitutes an extension to lower regularities of the invariance of the Gibbs measure by Bourgain~\cite{BO94}. Another consequence of Theorem~\ref{THM:mkdv} is the quasi-invariance of Gaussian measures under the flow of defocusing mKdV \eqref{mkdv}. Namely, the Gaussian measure
\begin{align}
	d\wt{\mu}_{s} = Z_{s}^{-1} \exp( -\tfrac{C_s}{2}\|q\|_{H^{s}(\T)}^{2})dq, \label{mutilde}
\end{align}
where $C_s>0$ is the constant defined in \eqref{Cs}, is mutually absolutely continuous with respect to the pushforward measure $\Phi^{\text{mKdV}}(t)_{\#} \wt{\mu}_s$. This follows from the equivalence of the measures $\mu_s$ and $\wt{\mu}_{s}$, which we prove in Section~\ref{SEC:equiv}, the mutual absolute continuity between $\rho_{s,R}$ and the restricted probability measures
\begin{align}
	Z^{-1}_{s,R} \ind_{\{ \|q\|_{L^{2}}^2 \leq R\}} d\mu_{s},
	\label{gauss-rest}
\end{align}
and the monotone convergence theorem (to take $R\to \infty$). This (only) extends the qualitative statement of quasi-invariance in \cite{PTV2} to lower-regularities for the mKdV equation.\footnote{Namely, we do not say anything about the transported density.} We refer to \cite{Tz, OTz2, DT2020, PTV2} and the references therein for further discussion on quasi-invariance of Gaussian measures for dispersive PDEs.

The presence of a mass cut-off in \eqref{GibbsEmeas} may be seen as somewhat unnatural given the defocusing nature of the equation. We recall that the defocusing Gibbs measure in \eqref{Gibbs2} can be rigorously constructed without a mass cut-off because of the sign-definiteness of the Hamiltonian \eqref{Hmkdv}.
In our setting, this translates to establishing that the integral involving $\mathcal{V}(\kk, q)$ on the right-hand side of \eqref{EsV} is non-negative, at least on the support of the measure $\mu_s$. Due to the involved nature of $A(\kk,q)$ and hence of $\mathcal{V}(\kk,q)$, it is not clear if this claim is true. Nonetheless, the result of Lemma~\ref{LEM:EMconv} verifies that this quantity is finite almost-surely. Moreover, we do succeed in constructing the measures in \eqref{GibbsEmeas} for every $R>0$, and thus also recovering the support of the Gaussian measure $\mu_s$; see Theorem~\ref{THM:mkdv}~(ii).

The result of Theorem~\ref{THM:mkdv} immediately combines with the Poincar\'e recurrence theorem \cite[p. 106]{ZhidPoin} to yield the following recurrence phenomena.

\begin{corollary}\label{COR:Poincare}
Let $\frac 12<s<1$, $\s<s-\frac 12$, and $R>0$. Then, there is a measurable set $\Sigma\subseteq H^{\s}(\T)$ such that $\rho_{s,R}(\Sigma)=1$ and for $u_0\in \Sigma$ and $u(t)$ the solution to \eqref{mkdv}, there exists $t_{n}\to \infty$ as $n\to \infty$ for which
\begin{align*}
\lim_{n\to \infty} \| u(t_n)-u_0\|_{H^{\s}(\T)}=0.
\end{align*}
\end{corollary}

Our analysis is also quite general and we expect similar results to hold true for many other completely integrable Hamiltonian PDEs, which are at least amenable to the analysis used in the method of commuting flows.
In particular, we show that an analogue of Theorem~\ref{THM:mkdv} holds for the defocusing cubic NLS equation on $\T$:
\begin{align}
i\dt q = -\dx^{2} q + 2|q|^2 q. \label{NLS}
\end{align}
See Section~\ref{SEC:NLS} for a discussion on the measures and the needed modifications and Theorem~\ref{THM:NLS} for a statement of the results. We do not provide full details for \eqref{NLS} to avoid repetition with the case of \eqref{mkdv}, and instead emphasize the connection of \eqref{mkdv} with the KdV equation which we discuss below. We also point out that our results for \eqref{NLS} (Theorem~\ref{THM:NLS}) and \eqref{KdV} (Theorem \ref{THM:kdv}) imply corresponding recurrence statements for solutions as in Corollary~\ref{COR:Poincare}, which further emphazises the generality of our approach.

\medskip

Our second main contribution is the construction of invariant measures for KdV supported at low regularity. To this end, we focus on initial data with mean zero and consider the base Gaussian measure $\mu_{s}^{0}$ obtained by removing the mass term in \eqref{gauss0}. Note that the results in Theorem~\ref{THM:mkdv} extend to the corresponding weighted measures $\rho_{s,R}^{0}$ with base Gaussian $\mu_{s}^{0}$, where the mKdV dynamics are restricted to mean zero data.
Here, we have access to the (corrected) Miura transform:
\begin{align*}
	B:L^2_{0}(\T)\mapsto H^{-1}_{0}(\T), \quad B(q)=\dx q + q^2-\|q\|_{L^2(\T)}^2,
\end{align*}
where $L^2_{0}(\T)$ and $H^{-1}_{0}(\T)$ are subspaces of $L^2(\T)$ and $H^{-1}(\T)$ restricted to mean-zero distributions, respectively. The Miura transform maps smooth solutions of defocusing mKdV to smooth solutions of KdV and
Kappeler-Topalov \cite{KTRiccati} showed that $B$ is a real-analytic isomorphism.
We define the probability measure
\begin{align*}
	\nu_{s,R}^0 =B_{\#}\rho_{s,R}^{0}, \quad \nu_{s,R}^0(A) = \rho_{s,R}^{0}(B^{-1}(A))
\end{align*}
for Borel measurable sets $A\subseteq H^{-1}_{0}(\T)$, and the map
\begin{align}
	\Phi^{\text{KdV}}_{0}(t)(w^0): = B\circ \Phi^{\text{mKdV}}(t)\circ B^{-1}(w^0) \label{intro:Kdvmeanzero}
\end{align}
for $w^0\in H^{-1}_{0}(\T)$.
The mean-zero assumption allows us to show that the KdV dynamics from \cite{KapTopKdv, KV} agree with those in \eqref{intro:Kdvmeanzero} and preserve the measures $\nu_{s,R}^0$.
A similar discussion applies to initial data with prescribed non-zero mean.

\begin{theorem}\label{THM:kdv}
Let $\tfrac{1}{2}<s<1$ and $R>0$. \\

	\noi
	{\rm(i)}
	The probability measures $\nu_{s,R}^{0}$, which are supported on $H_{0}^{\s}(\T)\setminus H_{0}^{s-\frac{3}{2}}(\T)$, for every $\s<s-\tfrac{3}{2}$, are invariant under the $\Phi_0^{\textup{KdV}}$-flow: namely,
	\begin{align*}
		\int_{H^{-1}_{0}(\T)} f( \Phi_0^{\textup{KdV}}(w))d\nu_{s,R}^{0}(w) = \int_{H^{-1}_{0}(\T)} f( w)d\nu_{s,R}^{0}(w) ,
	\end{align*}
	for all $t\in \R$ and $f\in L^1(H^{-1}_{0}(\T);d\nu_{s,R}^0)$.

	\noi{\rm(ii)} For $\al\in\R$, define the KdV data-to-solution map on $$H^{-1}_\al(\T) = \{q \in H^{-1}(\T): \, \textstyle{\int_{\T}} \,q(x) dx = \al\},$$ by
	$\Phi^{\textup{KdV}}_{\al}(t): = \tau_{-\al}\circ \Phi^{\textup{KdV}}_{0}(t)\circ \tau_{\al},$
	where $\tau_\al(q) = q-\al$. Then, the probability measure $\nu^{\al}_{s,R}:=(\tau_{-\al})_{\#}\nu_{s,R}^{0}=(\tau_{-\al}\circ B)_{\#}\rho_{s,R}^0$ is invariant under the flow $\Phi^{\textup{KdV}}_\al(t)$ for all $t \in\R$.

	\noi{\rm(iii)} For $\al\in\R$, the map $\Phi^{\textup{KdV}}_\al(t)$ agrees with those in \cite{KapTopKdv, KV} restricted to $H^{-1}_{\al}(\T)$.
\end{theorem}

We also view Theorem~\ref{THM:kdv} as an extension of the results concerning the invariance of the white noise measure on $\T$, formally given by
\begin{align}
	d\wt{\mu}_{0}:=Z^{-1} \exp( - M(q))dq, \label{white}
\end{align}
under the flow of KdV.
Rigorously, the white noise is the Gaussian measure on periodic distributions with the identity as covariance operator. As such, it is supported on $H^{\s}(\T)$ for $\s<-1/2$. As $M(q)$ is also a conserved quantity for KdV \eqref{KdV}, we expect \eqref{white} to be invariant under the flow of KdV, which was rigorously proved by Quastel, Valk\'{o}, and Oh~\cite{QV, OhWhite, OhWhite2, OQV}; see also~\cite{KMV}. The invariant measures we construct in Theorem~\ref{THM:kdv} are supported on $H^\s(\T) \setminus H^{s-\frac32}(\T)$ for $\s<s-\frac32$ and $\frac12 < s <1$,
 thus covering the remaining well-posedness range between the sharp well-posedness results in $H^{-1}(\T)$ \cite{KapTopKdv, KV} and the sharp well-posedness in $H^{-\frac{1}{2}}(\T)$ for which the data-to-solution map is uniformly continuous on bounded sets \cite{KPV4, CCT}.

We emphasise that the invariant measures we construct for the defocusing mKdV, defocusing NLS, and  KdV cover the lowest possible regularities for which these equations are well-posed. Indeed, their support is rougher than that of the Gibbs measures for mKdV and NLS and of the white noise for KdV.

\begin{remark}\rm

 Heuristic computations suggest that $\nu_{s,R}^{0}$ is absolutely continuous with respect to the Gaussian measure \begin{align}
 Z_{s,R}^{-1}\ind_{\{ \| w\|_{H_{0}^{-1}}^{2}\leq R\}}\exp \big( -\tfrac{1}{2} \| |\partial|^{-1} \mathfrak{m}_{s}(-i\dd)w\|_{L^2_{0}}^{2}\big) dw.
\end{align}
In particular, by Lemma~\ref{LEM:ms}, the Gaussian measure above with covariance operator $|\dd|^{-2}\mathfrak{m}_{s}(-i\dd)^{2}$ is equivalent to the Gaussian measure
\begin{align*}
d\wt{\mu}_{s-1} = Z_{s}^{-1} \exp( -\tfrac{C_s}{2}\|w\|_{H_{0}^{s-1}(\T)}^{2})dw,
\end{align*}
with the same constant $C_{s}$ as in \eqref{Cs}.
It would be of interest to rigorously prove this result, perhaps using ideas from \cite{CMK, KTRiccati}.
\end{remark}

Before discussing further details about the proof, we make the following remarks regarding the regularity restriction on $s$ in Theorem~\ref{THM:mkdv}.

\begin{remark}\rm \label{RMK:above}
	Our interest in the range $\frac12<s<1$ is due to the low regularity of the conserved quantities considered and corresponding support of the measures $\rho_{s,R}$, which is strictly rougher than that of the Gibbs measure \eqref{Gibbs2} for mKdV.
	We believe our method can be extended to larger values of $s$, by considering further differences in defining $\mathcal{A}(\kk,q)$ in \eqref{Acal:intro}, and introducing additional cut-offs depending on conserved quantities of the form $\mathcal{E}_{s}(q)$ at regularities $s-1/2-\eps$, for small $\eps>0$. To keep the exposition short, we do not pursue this direction here.
\end{remark}

\begin{remark}\rm \label{RMK:belowL2}
	An interesting open question would be to extend the construction and invariance results in Theorem~\ref{THM:mkdv} to the singular range $0<s\leq 1/2$.
	The resulting measures would be supported strictly below $L^2(\T)$.
	This presents many additional challenges: Firstly, one would have to suitably extend the results in Section~\ref{SEC:operators} beyond $L^2(\T)$-potentials.
 Secondly, there is the question of dynamics:
 there is no known well-posedness result on the $L^2$-based Sobolev spaces $H^{\s}(\T)$ for $\s<0$, which contain the support of the intended measure. In fact, \eqref{mkdv} is ill-posed below $L^2(\T)$~\cite{Molinet}. Thus, it is necessary to consider the following renormalized mKdV equation:
	\begin{align}
		\dt q=-\dx^{3} q\pm 6 \big( q^2 -2 M(q) \big)  \dx q, \label{renormmkdv}
	\end{align}
	first introduced by Bourgain~\cite{BO2}. In $L^2(\T)$, mKdV~\eqref{mkdv} and the renormalized mKdV~\eqref{renormmkdv} are equivalent, as solutions of the former can be related to solutions of the latter via the gauge transformation
	$q(t,x) \mapsto q\big(t, x\pm 12 t M(q)  \big). $

	Outside of $L^2(\T)$, the best known result in this direction is by Kappeler-Molnar~\cite{KapMol} using the complete integrability to prove that  defocusing mKdV \eqref{renormmkdv} is locally well-posed in the Fourier-Lebesgue spaces $\FL^{p}(\T)$ for $2<p<\infty$ and globally for small data (this small data result extends to the focusing case too).
We point out that the samples of $\mu_s$ for $0<s\leq 1/2$ belong to $\FL^{p}(\T)$ almost surely provided that $p>1/s$.
Therefore, if a candidate invariant measure is absolutely continuous with respect to $\mu_s$, the rigorous invariance of the measure would follow from local well-posedness of \eqref{renormmkdv} in $\FL^p(\T)$ for $p>1/s$.
Unfortunately, the local well-posedness result in \cite{KapMol} does not seem to be suitable as the time-of-existence of the solutions does not depend in an explicit way on the size of the initial data, so it is unclear if one could adapt Bourgain's invariant measure argument~\cite{BO94}.

\end{remark}

\begin{remark}\rm \label{RMK:focusing}
	For the focusing mKdV equation, we expect to be able to construct corresponding measures $\rho_{s,R}$ using the series expansion \eqref{alpha}. However, as we previously discussed, this would enforce $\kk\geq 100 (1+R)$, which we find unsatisfactory. In any case, the invariance of these measures would then follow from the well-posedness in \cite{Forlano} and the method delineated in this paper.

\end{remark}

\begin{remark}\rm \label{RMK:taming}

A minor modification of the arguments in this paper allow us to construct the wider class of probability measures
\begin{align}
d\,\mathfrak{p}_{s,a}:= Z_{s}^{-1}\, \exp\bigg(- \int_1^{\infty} \kk^{2s} \mathcal{V}(\kk, q)  \, d\kk - \| q\|_{L^2}^{2a} \bigg) d\mu_{s}, \label{Gibbstamed}
\end{align}
provided that $a>a(s)$ is sufficiently large, which are absolutely continuous with respect to the Gaussian measures $\mu_s$. In view of the inequality
\begin{align*}
\ind_{\{|x|\leq R\}} \leq e^{R^{a}}e^{-|x|^{a}}, \quad x\in \R, \quad a,R\geq 0,
\end{align*}
we see that the construction of $\mathfrak{p}_{s,a}$ implies the construction of $\rho_{s,R}$ in \eqref{GibbsEmeas}.
%$a>\max(5, \tfrac{3-2s}{2(1-s)})$
Moreover, the measures $\mathfrak{p}_{s,a}$ are also invariant under the defocusing mKdV flow.
An advantage of these measures over those with an $L^2$-cut-off is that they individually see the whole support of the Gaussian measures $\mu_s$. More precisely, $\supp \mathfrak{p}_{s,a}=\supp \mu_s$ for each such fixed $a$, which should be compared to \eqref{supppsr} in Theorem~\ref{THM:mkdv} (ii).

In the context of Gibbs measures for focusing nonlinear wave-type equations, a subtraction of a power of the $L^2$-norm in the exponential of the density (or a renormalised version of it) rather than an $L^2$-cut-off is necessary due to the lack of $L^2$-conservation. See \cite{BO99, OOT1, OOT2} and the references therein.

\end{remark}

\begin{remark}\rm
	In \cite{KT}, Koch-Tataru constructed conserved quantities for mKdV \eqref{mkdv} on the line which control the $H^\s(\R)$-norm for any $\s>-\frac12$, for small initial data. If one can extend their results to the periodic setting and for large data, it may be of interest to study the candidate weighted Gaussian measures associated with these conserved quantities.

\end{remark}

We now discuss in more detail some elements of the proof of Theorem~\ref{THM:mkdv}. We split the argument into two parts: (i) construction of the probability measures \eqref{GibbsEmeas}, and (ii) invariance under the dynamics of the defocusing mKdV~\eqref{mkdv}.

In order to construct the measures \eqref{GibbsEmeas}, we need to prove that the density is integrable with respect to the Gaussian measure $\mu_s$ in \eqref{gauss0}. Spurred on by the construction of the $\Phi_{3}^{4}$-measure given by Barashkov-Gubinelli~\cite{BG}, the Bou\'e-Dupuis variational formula \cite{BD, Ust} has proved to be a very effective tool at constructing probability measures with respect to a base Gaussian measure. This reduces the integrability of the density to a stochastic control problem which can be handled via suitable estimates on the logarithm of the density. In our setting, this translates to bounds on the remainder piece $\mathcal{V}(\kk,q)$ which are both valid for any $\kk\geq 1$, and come with enough decay in $\kk$ for convergence of the integral in~\eqref{GibbsEmeas}.

The bounds on $\mathcal{A}(\kk,q)$ come from the recent work of the second author, Killip, and Vi\c{s}an~\cite{FKV} on the invariance of the Gibbs measure for the defocusing mKdV \eqref{mkdv} on the line. The key point is that these bounds are proved non-perturbatively; thus requiring no relation between $\kk$ and the $L^2$-norm of $q$. In fact, the bounds in \cite{HGKV, Forlano} are of no use.

	As to the dynamical problem of invariance, we are in the fortunate position of having deterministic global well-posedness on the support of our measures. It thus remains to prove the invariance, for which we need an approximating flow. As we rely on the well-posedness proved from the method of commuting flows, it is natural to first consider the approximate Hamiltonian flow of the quantity $H^{\text{mKdV}}_{\kk}$, $\kk\geq 1$, defined in \eqref{Hvk}. The heuristic at play is that the $H^{\text{mKdV}}_{\kk}$-flow (i) is easily solvable in $L^2(\T)$, (ii) commutes with the $H^{\text{mKdV}}$-flow, and (iii) converges to the $H^{\text{mKdV}}$-flow as $\kk\to \infty$. See Proposition~\ref{P:HK well}.

	The main idea is then to prove that the $H^{\text{mKdV}}_{\kk}$-flow leaves the measures $\rho_{s,R}$ invariant. We proceed in a standard way by considering a truncated version of the $H^{\text{mKdV}}_{\kk}$-flow. Unfortunately, at this truncated level, we cannot guarantee that $\mathcal{E}_{s}$ is conserved.
	To resolve this, we use ideas from \cite{NORBS, TV1,TV2}, and instead show that these quantities are almost conserved in the sense that asymptotically we have conservation in the truncation limit. See Proposition~\ref{PROP:asym-cons}. Due to the nicer structure of $\mathcal{E}_{s}$ and of the tameness of the truncated $H^\text{mKdV}_\kk$-flow, our version of this is much stronger: it is completely deterministic and for all times, not just at time zero.

The remaining of the paper is organized as follows. In Section~\ref{SEC:preliminary}, we introduce the relevant notation and the structures derived from the complete integrability of the equation needed for the construction and invariance of the measures.
Section~\ref{SEC:construct} is devoted to the construction of the measures $\rho_{s,R}$, while Section~\ref{SEC:invariance} proves their invariance under the mKdV dynamics. In Subsection~\ref{SEC:kdv}, we prove Theorem~\ref{THM:kdv} on invariant measures for KdV. Lastly, in Section~\ref{SEC:NLS}
we construct and prove the invariance of probability measures analogous to the defocusing mKdV ones under the flow of the defocusing cubic NLS.

\section{Preliminaries}\label{SEC:preliminary}
\subsection{Notations}
We write $A\les B$ to denote that there exists $C>0$ such that $A \leq CB$, and $A\ll B$ when $A\leq C B$ with $0<C<\frac12$. We use $\les_\al$ to note the dependence of the implicit constant on the parameter $\al$.
Our conventions for the Fourier transform on the line are
\begin{align*}
	\ft f(\xi)=\tfrac{1}{\sqrt{2\pi}}\int_{\R} e^{-i\xi x}f(x)dx \quad \text{and}\quad f(x)=\tfrac{1}{\sqrt{2\pi}}\int_{\R}e^{i\xi x}\ft f(\xi)d\xi,
\end{align*}
while on $\T$ we set
\begin{align*}
	\ft f(\xi)= \int_{0}^{1} e^{-i\xi x} f(x)dx \quad \text{and} \quad f(x)=\sum_{\xi \in 2\pi \Z}\ft f(\xi)e^{i\xi x}.
\end{align*}

\noi
In this context, Plancherel's theorem takes the form
\begin{align*}
	\|f\|_{L^2(\T)}^{2}=\int_0^{1}|f(x)|^2 dx= \sum_{\xi \in 2\pi \Z} |\ft f(\xi)|^{2}.
\end{align*}
We have $\ft{\dx f}(\xi)=i\xi \ft f(\xi)$, and we define the $L^2$-based Sobolev spaces $H^s$, $s\in \R$, via
\begin{align*}
	\|f\|_{H^{s}(\T)}^{2}=\sum_{\xi \in 2 \pi \Z} \jb{\xi}^{2s} |\ft f(\xi)|^{2} \quad \text{and} \quad  \|f\|_{H_{\kk}^{s}(\T)}^{2}= \sum_{\xi \in2 \pi \Z} (4\kk^2+\xi^2)^{s} |\ft f(\xi)|^{2},
\end{align*}
where $\kk>0$. The definitions on the line are analogous.
Note that by Cauchy-Schwarz,
\begin{align}
	\| f\|_{L^{\infty}(\T)} \les \| \ft f(\xi)\|_{\l^{1}_{\xi}}   \leq  \| (4\kk^2+\xi^2)^{-\frac{1}{2}}\|_{\l^2_{\xi}}   \|f\|_{H^{1}_{\kk}(\T)} \les \kk^{-\frac{1}{2}}\|f\|_{H^{1}_{\kk}(\T)} . \label{LinftyH1k}
\end{align}

\noi
For $1\leq p\leq \infty$, we define the spaces $W^{s,p}(\T)$ via the norm
\begin{align*}
	\| f\|_{W^{s,p}(\T)} =\| \mathcal{F}^{-1}\{ \jb{\xi}^{s} \ft f(\xi)\}\|_{L^{p}(\T)}.
\end{align*}
In the remaining of the paper we will omit $\T$ when indexing norms, for simplicity.
Also, we will use $\partial$ instead of $\partial_x$ when clear from the context.

\begin{lemma}\label{LEM:interp}
	Given $s>\tfrac{1}{2}$ and $\kk\geq 1$, we have
	\begin{align*}
		\| (2\kk \pm \dd)^{-1}q \|_{L^{\infty}} \les \kk^{-\frac{1+\ta}{2}} \|q\|_{L^2}^{1-\ta}\|q\|_{H^{s}}^{\ta},
	\end{align*}
	for any $0\leq \ta\leq 1$. Moreover, given any $0 \leq \s <\tfrac{1}{2}$, we have
	\begin{align*}
		\| (2\kk \pm \dd)^{-1}q\|_{L^{\infty}} \les \kk^{-\s-\frac{1}{2}}\|q\|_{H^{\s}}.
	\end{align*}
\end{lemma}
\begin{proof}
	On the one hand, Sobolev embedding implies that
	\begin{align}
		\|  (2\kk \pm \dd)^{-1}q\|_{L^{\infty}} \les \| (2\kk \pm \dd)^{-1}q \|_{H^{\frac{1}{2}+\eps}} \les \kk^{-1}\| q\|_{H^s}. \label{interp1}
	\end{align}
	provided that $\eps>0$ is sufficiently small.
	On the other hand, Cauchy-Schwarz gives
	\begin{align*}
		\|  (2\kk \pm \dd)^{-1}q\|_{L^{\infty}} \les \|  (4\kk^2+\xi^2)^{-\frac{1}{2}} \ft q(\xi) \|_{\l^{1}_{\xi}}   \les\|  (4\kk^2+\xi^2)^{-\frac{1}{2}}\|_{\l^2_\xi}  \|q\|_{L^2}.
	\end{align*}
	By splitting the above summation into the regions $|\xi|\les 2\kk $ and $|\xi|\gg 2\kk$, we get
	\begin{align}
		\|  (2\kk \pm \dd)^{-1}q\|_{L^{\infty}}  \les \kk^{-\frac{1}{2}}\|q\|_{L^2}. \label{interp2}
	\end{align}
	Now the first inequality follows by interpolation between \eqref{interp1} and \eqref{interp2}.
	The proof of the second estimate follows the same ideas as in \eqref{LinftyH1k}.
\end{proof}

\subsection{Structures for mKdV}\label{SEC:operators}

In this section, we recall important facts regarding the Green's function of the Lax operator for the defocusing mKdV~\eqref{mkdv}:
\begin{align*}
	\mathcal{L}=\L_{q}=\begin{bmatrix}
		-\dd& q\\
		-q &  \dd \\
	\end{bmatrix}.
\end{align*}
For $q\in L^2(\T)$, we view $\L$ as an unbounded operator with domain $H^{1}(\R)\times H^1(\R)$ which is anti-self adjoint on $L^2(\R)\times L^2(\R)$.
Thus the resolvent $R(\kk):=(\L+\kk)^{-1}$ exists for all $\kk\geq \tfrac{1}{2}$ and it is jointly analytic on $\{ (\kk, q):\, \kk\geq \tfrac{1}{2}, \,\, q\in L^2(\T)\}$. Moreover, it admits an integral kernel, which we call the Green's function $G(x,y;\kk,q)$, that is a $2\times 2$ matrix satisfying
\begin{align}
	G(x,y;\kk;q)-G_0 (x,y;\kk)=\jb{\dl_x, [R(\kk)-R_0(\kk)]\dl_y}, \label{Gcts}
\end{align}
where $G_{0}$ is the Green's function of the free resolvent $R_{0}(\kk):=(\L_{0}+\kk)^{-1}$, $q\equiv 0$, and is given by
\begin{align*}
	G_{0}(x,y;\kk)=e^{-\kk|x-y|}
	\begin{bmatrix}
		\ind_{\{ x<y\}} & 0 \\
		0 & \ind_{\{ y<x\}}
	\end{bmatrix}.
\end{align*}
Note that $G$ depends on the potential $q$ and, when the context is clear, we simply write $G(x,y;\kk)$ in place of $G(x,y;\kk, q)$. The same comment applies to other related quantities such as the components of the Green's function in \eqref{Gsyms} and the elements in \eqref{rdefn}.
Furthermore, $G$ is a continuous function of $(x,y)\in \R^2$, $x\neq y$, $G(x,y;\kk)\in L^{\infty}(\R^2)$, and all components of $G(x,y;\kk)$ decay to zero as $|x|\to\infty$ for fixed $y\in\R$.

As the Lax operator satisfies
\begin{align*}
	\L^{\ast}=-\L = \begin{bmatrix}
		0 & 1 \\
		1 & 0
	\end{bmatrix} \L \begin{bmatrix}
		0 & 1 \\
		1 & 0
	\end{bmatrix},
\end{align*}
the entries of $G(x,y;\kk)$ enjoy the following symmetry properties:
\begin{equation}
	\begin{gathered}
		G_{11}(y,x;\kk)=G_{22}(x,y;\kk),\\
		G_{12}(y,x;\kk)=G_{12}(x,y;\kk), \quad
		G_{21}(y,x;\kk)=G_{21}(x,y;\kk).
	\end{gathered}\label{Gsyms}
\end{equation}
Furthermore,
\begin{align}
	\det G(x,y;\kk) = 0 \quad \text{for all} \quad x,y\in \R^2, x\neq y. \label{detG}
\end{align}
This follows from showing that the gradient of $\det G(x,y;\kk)$ in the variables $x,y$ vanishes and then appealing to the decay at $\pm \infty$. See \cite{FKV} for details.

By \eqref{Gcts}, we define the following linear combinations of the elements of $G$ on the diagonal:
\begin{align}
	\g(x; \kk, q)&: = (G-G_0)_{11}(x,x;\kk, q)+(G-G_0)_{22}(x,x;\kk, q), \nonumber\\
	g_{+}(x; \kk,q)& := G_{21}(x,x;\kk, q)+G_{12}(x,x;\kk, q),\nonumber\\
	g_{-}(x; \kk, q)&: = G_{21}(x,x;\kk, q) - G_{12}(x,x;\kk, q). \label{rdefn}
\end{align}

\noi
We will omit the dependence of $\g,g_+,g_-$ on all or some of their parameters when clear from context. We can explicitly write the quadratic in $q$ contributions of $\g(\kk,q)$ as follows:
\begin{align}
	\g^{[2]}(\kk, q):=-2\,\tfrac{q}{2\kk-\dd}\cdot \tfrac{q}{2\kk+\dd}, \label{G2}
\end{align}
with the shorthand $\frac{q}{2\kk \pm \dd}=(2\kk\pm\dd)^{-1}q$.
Also, let $\g^{[\geq 4]}(\kk):=\g(\kk)-\g^{[2]}(\kk)$.
Lastly, we define $A(\kk, q)$ as follows
\begin{align}\label{defA}
	A(\kk, q) := \int_\T \rho(\kk, q) \, dx,  \qquad
	\rho(\kk, q) := \frac{q \, g_{-}(\kk, q)}{2+\g(\kk, q)}.
\end{align}
The conservation of $A(\kk,q)$ for smooth solutions of mKdV follows from \eqref{derivA}, \eqref{Hmkdv}, and \eqref{Poisson}; see \cite{Forlano}.

We now pause and make an important remark in order to dispel potential confusion arising from our convention.
We stress that we are viewing the operators, such as $\mathcal{L}$ and $R(\kk)$, as operators on the Hilbert space $L^2(\R)$ with a periodic potential $q$, and not as acting on $L^2(\T)$. This is important since the construction of invariant measures independently of the mass-cut-off $R$  relies crucially on the macroscopic definition \eqref{defA} for $A(\kk,q)$. However, this seems difficult to justify from the series expansion for $A(\kk, q)$ if we had viewed these operators as acting on $L^2(\T)$.
The main source of the difficulty is in establishing an analogue of \eqref{detG}, which is a key step in this argument, for potentials in $L^2(\T)$; see the proof of Lemma 4.1 in \cite{HGKV}.
Indeed, we expect \eqref{detG} to not be true in the purely periodic setting since it is not true for the periodic free Green's function
\begin{align}
	\mathcal{G}_{0}(x,y;\kk) = \frac{1}{1-e^{-\kk}}
	\begin{bmatrix}
		e^{\kk( x-y-\ceil{x-y})} & 0 \\
		0 & e^{-\kk(x-y-\floor{x-y})}
	\end{bmatrix} \label{periodG}
\end{align}
where $\ceil{x}$ and $\floor{x}$ denote the smallest integer $n$ such that $x\leq n$, or the largest integer $m$ such that $m\leq x$, respectively.
While we may follow the same argument and show that the determinant of the periodic Green's function is independent of $x\neq y$, it may depend upon $q$ and this would need to be included in the ansatz for the measures.

Also, note that because of our convention, the operators $(2\kk\pm \dd)^{-1}$ appearing in \eqref{G2} are defined via their action on $L^2(\R)$. However, since $q$ is periodic, we have
\begin{align*}
	\frac{q}{2\kk \pm \dd}(x) &= \int_{0}^{1} \bigg( \sum_{\xi \in 2\pi \Z} G_{0}(x,y+\xi;2\kk)\bigg) q(y)dy = \int_{0}^{1} \mathcal{G}_{0}(x,y;2\kk) q(y)dy,
\end{align*}
where $\mathcal{G}_{0}$ is defined in \eqref{periodG}. Note that the Fourier coefficients of the operator $(2\kk\pm\dd)^{-1}$ acting on $L^2(\T)$ are $(2\kk \pm i\xi)^{-1}$. Thus, we may view the operators in \eqref{G2} in the usual way as operators on $L^2(\T)$ whenever they act on a periodic function. The other operator of importance that will appear in this way is $\mathsf{R}_0 (\kk) := (\kk^2-\partial^2)^{-1}$, for which we have
\begin{align*}
	\mathsf{R}_0 (\kk) q(x) =\int_{\R} \frac{1}{2\kk}e^{-\kk |x-y|}q(y)dy= \int_{0}^{1} \bigg( \sum_{\xi \in 2\pi \Z}  \frac{1}{2\kk} e^{-\kk |x-y+\xi|}\bigg) q(y)dy,
\end{align*}
with periodic Fourier coefficients $(\kk^2+\xi^2)^{-1}$. We apply this remark at all further relevant instances in this paper without further discussion.

For $\kk \gg 1+\|q\|_{L^2(\T)}^2$, the following properties of $\g(\kk)$ and $g_{-}(\kk)$ are essentially proven in \cite{Forlano} with the ideas in \cite[Section 6]{KV}; see also \cite{HGKV}. For arbitrary $\kk\geq \tfrac{1}{2}$, the identities remain true by analytic continuation.
However, we need estimates which are valid for all $\kk\geq \tfrac{1}{2}$, and these were proven in a non-perturbative fashion in \cite{FKV}; see also \cite{KMV} for the case of KdV.

\begin{proposition}[Properties of $\g$, and $g_{-}$]\label{PROP:Gamma}
	Let $q\in L^2(\T)$ and $\kk\geq \tfrac{1}{2}$. Then:

	\noi\textup{(i)} We have $\g(\kk, q)\in H^{1}_{\kk}(\T)$, $g_{-}(\kk, q)\in H^{2}_{\kk}(\T)$, and the following estimates hold
	\begin{align}
		\| \g(\kk, q)\|_{H^{1}_{\kk}(\T)} & \les \kk^{-\frac{1}{2}} \|q\|_{L^2(\T)}^{2} [1+\kk^{-2}\|q\|_{L^2(\T)}^{4}], \label{gbd} \\
		\| \g^{[\geq 4]}(\kk, q)\|_{H^{1}_{\kk}(\T)}& \les \kk^{-\frac{3}{2}}\|q\|_{L^2(\T)}^{4} [1+\kk^{-1}\|q\|_{L^2}^{2}]\label{gbd2}, \\
		\|g_{-}(\kk, q)\|_{H^{2}_{\kk}(\T)}& \les \kk\|q\|_{L^2(\T)} +\kk^{-2}\|q\|_{L^2(\T)}^{5}. \label{rbdL}
	\end{align}
	In particular, if $q\in C^{\infty}(\T)$, then $\g(\kk,q), g_{-}(\kk,q)\in C^{\infty}(\T)$.

	\noi\textup{(ii)} The mappings $q\mapsto \g(\kk,q)$ and $q\mapsto g_{-}(\kk,q)$ are Lipschitz in the following sense:
	\begin{align}
		&\|\g(\kk,q_1)-\g(\kk,q_2)\|_{H^1_{\kk}(\T)} \notag \\
		&\quad\les \kk^{-\frac{1}{2}}\|q_1-q_2\|_{L^2(\T)} ( \|q_1\|_{L^2(\T)}+\|q_2\|_{L^2(\T)}) \bigl[ 1+\kk^{-\frac12} (\|q_1\|_{L^2(\T)} + \|q_2\|_{L^2(\T)})\bigr]^6, \label{g lip}\\
		&\|\g^{[\geq 4]}(\kk,q_1)-\g^{[\geq 4]}(\kk,q_2)\|_{H^1_{\kk}(\T)} \notag \\
		&\quad\les \kk^{-\frac{3}{2}}\|q_1-q_2\|_{L^2(\T)} ( \|q_1\|_{L^2(\T)}+\|q_2\|_{L^2(\T)}) \bigl[ 1+\kk^{-\frac12} (\|q_1\|_{L^2(\T)} + \|q_2\|_{L^2(\T)})\bigr]^6, \label{g lip4}\\
		&\|g_{-}(\kk, q_1) - g_{-}(\kk,q_2)\|_{H^2_\kk (\T)}  \notag \\
		&\quad \les  \kk \|q_1-q_2\|_{L^2(\T)}( \|q_1\|_{L^2(\T)}+\|q_2\|_{L^2(\T)})^2 [1+\kk^{-\frac{1}{2}}(\|q_1\|_{L^2(\T)}+\|q_2\|_{L^2(\T)})]^{6} \label{rLip},
	\end{align}
	for any $q_1, q_2\in L^2(\T)$.

	\noi\textup{(iii)} The following differential identities hold in the sense of distributions:
	\begin{align}
		\dx \g(\kk)& = 2q g_+(\kk)  \quad \text{and} \qquad
		\dx g_{-}(\kk)=-2\kk g_+(\kk), \label{Grderiv}
	\end{align}
	and also
	\begin{align}
		g_{-}(\kk)&= 4\kk \mathsf{R}_{0}(2\kk)( q\g(\kk)+q). \label{prformula}
	\end{align}

	\noi\textup{(iv)} The pointwise inequality holds:
	\begin{align}
		1+\g(\kk)> 0. \label{gammap1}
	\end{align}

	\noi\textup{(v)} Lastly, we have that
	\begin{align}
		\frac{\dl A(\kk,q)}{\dl q} &= g_{-}(\kk, q), \label{derivA}\\
		\frac{\dl g_{-}(x;\kk,q)}{\dl q}(y) & =G_{11}(x,y;\kk)^2 +G_{22}(x,y;\kk)^2  - G_{12} (x,y;\kk)^2- G_{21}(x,y;\kk)^2. \label{rderivq}
	\end{align}

\end{proposition}

The bounds in (i) are proved as follows: for large $\kk$, the resolvents $R(\kk)$ may be expanded as an infinite series for which good bounds hold. These can then be transferred to $R(\kk)$ for any $\kk\geq \tfrac{1}{2}$ by the first resolvent identity. The bounds in (ii) follow similarly by also using the second resolvent identity.
Whilst \eqref{g lip4} is not explicitly written in \cite{FKV}, it follows from the estimates there and similar arguments as required for \eqref{g lip}. As for \eqref{gammap1}, as $\mathcal{L}$ is anti-self-adjoint, we have $\jb{\phi, \mathcal{L} \phi}=0$
for all real-valued $\phi \in H^{1}(\R)\times H^{1}(\R)$ and hence
\begin{align}
\jb{\phi, R(\kk)\phi} = \kk \jb{ R(\kk)\phi, R(\kk)\phi}. \label{resolv1+}
\end{align}
The result then follows by choosing $\phi$ appropriately as an approximate identity and taking limits in \eqref{resolv1+}. See \cite{FKV} for details.
We note that \eqref{rderivq} follows from \eqref{rdefn} and the following consequence of the second resolvent identity
\begin{align*}
	\frac{d}{d\ta} G(x,y; \kk, q+\ta f) \bigg\vert_{\ta=0}  = -\int_{\R} G(x,z;\kk) \begin{bmatrix}
		0 & f(z) \\
		-f(z) & 0
	\end{bmatrix} G(z,y;\kk)dz.
\end{align*}

We now discuss the approximating flows and dynamics.
Given $\kk\geq \tfrac{1}{2}$ and $q\in L^2(\T)$, let
\begin{align}
	H^{\text{mKdV}}_{\kk}(q)  = 4\kk^2\int_{0}^{1} q(x)^2 dx-4\kk^3 A(\kk,q). \label{Hvk}
\end{align}

\begin{proposition}\label{P:HK well}
	Fix $\kk\geq \tfrac{1}{2}$.  Then, the Hamiltonian flow induced by $H^{\textup{mKdV}}_{\kk}$,
	\begin{align}
		\dt q = 4\kk^2 \dx q - 4\kk^3 \dx g_{-}(\kk,q) \label{Hkflow}
	\end{align}
	is globally well-posed on $L^2(\T)$ and commutes with the mKdV flow. Moreover,
	for fixed $q^0\in L^2(\T)$ and $T>0$,  let $\Phi^{\textup{mKdV}}(t)(q^0)$ and $\Phi^{\textup{mKdV}}_{\kk}(t)(q^0)$ denote the global solutions to \eqref{mkdv} and \eqref{Hkflow} with initial data $q^0$, respectively. Then,
	\begin{align}
		\lim_{\kk \to \infty} \sup_{|t|\leq T}\bigl\|\Phi^{\textup{mKdV}}_{\kk}(t)(q^0)-\Phi^{\textup{mKdV}}(t)(q^0)\bigr\|_{L^2 (\T)}=0.  \label{mkdvapprox}
	\end{align}
\end{proposition}

The proof of Proposition~\ref{P:HK well} can be found in \cite{FKV}, and in \cite{Forlano} for large~$\kk$. We briefly detail the arguments here.
Local well-posedness follows by the Cauchy-Lipschitz theorem using \eqref{rLip}, and global well-posedness in $L^2(\T)$ is a consequence of the conservation of the $L^2(\T)$-norm which can be proven by using \eqref{Grderiv}. Finally, the convergence of the flows \eqref{mkdvapprox} can be argued as in \cite{Forlano}, with the additional observation that we can recover $q$ from $g_{-}(\kk,q)$ and $\g(\kk,q)$ via the formula
\begin{align*}
	q=\tfrac{1}{4\kk}\tfrac{ (4\kk^2-\dd^2) g_{-}(\kk)}{1+\g(\kk)},
\end{align*}
where the maps $q\mapsto \tfrac{1}{4\kk}g_{-}(\kk,q)$ and $q\mapsto \g(\kk,q)$ are diffeomorphisms from $L^2(\T)$ to its range in $H^{2}(\T)$, and from $L^2(\T)$ to its range in $H^{1}(\T)$, respectively.

Lastly, we further expand $A(\kk, q)$ and establish some estimates on this quantity.
Using that $g_{-}(\kk,q)= 4\kk \mathsf{R}_{0}(2\kk)[q \g(\kk)+q]$ and $\tfrac{2}{2+\g}=1-\tfrac{\g}{2+\g}$, we have
\begin{align*}
	\rho(\kk,q) &= 2\kk q  \mathsf{R}_{0}(2\kk)q +\frac{2\kk}{2+\g(\kk)}\big\{ 2q \mathsf{R}_{0}(2\kk)[q \g(\kk)]-q \g(\kk)  \mathsf{R}_{0}(2\kk) q \big\} \\
	& =: \rho^{[2]}(\kk,q)+\rho^{[\geq 4]}(\kk,q).
\end{align*}
We see from \eqref{defA} that
\begin{align}
	A^{[2]}(\kk,q) = \int_{\T} \rho^{[2]}(\kk,q)dx \qquad \text{and} \qquad A^{[\geq 4]}(\kk,q) = \int_{\T} \rho^{[\geq 4]}(\kk,q)dx.
	\label{Adecomposition}
\end{align}
We need to expand $\rho^{[\geq 4]}(\kk, q)$ further. Proceeding as above, we have
\begin{align*}
	\rho^{[\geq 4]}(\kk,q)
	& = \kk \big\{ 2q \mathsf{R}_{0}(2\kk)[q \g^{[2]}(\kk)]-q \g^{[2]}(\kk)  \mathsf{R}_{0}(2\kk) q \big\} \\
	& \hphantom{X} + \bigg\{ \frac{2\kk}{2+\g(\kk)}\big\{ 2q \mathsf{R}_{0}(2\kk)[q \g^{[\geq 4]}(\kk)]-q \g^{[\geq 4]}(\kk)  \mathsf{R}_{0}(2\kk) q \big\} \\
	& \hphantom{XXXX} -\frac{\kk \g(\kk)}{2+\g(\kk)}\big\{ 2q \mathsf{R}_{0}(2\kk)[q \g^{[2]}(\kk)]-q \g^{[2]}(\kk)  \mathsf{R}_{0}(2\kk) q \big\}\bigg\} \\
	&=: \rho^{[4]}(\kk,q)+\rho^{[\geq 6]}(\kk,q).
\end{align*}
Since $\mathsf{R}_{0}(2\kk)$ is symmetric, upon integration and using \eqref{G2}, we obtain
\begin{align}
	A^{[4]}(\kk, q) &:= \int_{\T} \rho^{[4]}(\kk,q)dx
	 = -2  \int_\T \kk q \mathsf{R_0}(2\kk) \big[ q \cdot \tfrac{q}{2\kk - \partial} \cdot \tfrac{q}{2\kk+ \partial}\big] \, dx.
	\label{A4}
\end{align}
It will be useful later to think of \eqref{A4} as defining a multilinear operator. In this vein, Cauchy-Schwarz inequality implies the following estimate.

\begin{lemma}
	Given $\kk\geq \tfrac{1}{2}$, it holds that
	\begin{align}
		\bigg| \int_{\T} \kk  q_1 \mathsf{R}_{0}(2\kk)[ q_2 \cdot  \tfrac{q_3}{2\kk-\dd} \cdot \tfrac{q_4}{2\kk+\dd}]  dx\bigg| & \les \kk^{-1} \|q_1\|_{L^2}\|q_2\|_{L^2} \| \tfrac{q_3}{2\kk-\dd}\|_{L^\infty}\| \tfrac{q_4}{2\kk+\dd}\|_{L^{\infty}}, \label{A4general}
	\end{align}
	for any $q_j\in L^{2}(\T)$, $j=1,\ldots, 4$.

\end{lemma}

We also define
\begin{align}
	A^{[\geq 6]}(\kk, q):= \int_{\T} \rho^{[\geq 6]}(\kk,q)dx, \label{Ageq6}
\end{align}
which we further write out as
$A^{[\geq 6]}(\kk, q) = A^{[\geq 6]}_{1}(\kk, q)+A^{[\geq 6]}_{2}(\kk, q),
$
where
\begin{align}
	A^{[\geq 6]}_{1}(\kk, q)& =-\int_{\T}\frac{\kk \g(\kk)}{2+\g(\kk)}\big\{ 2q \mathsf{R}_{0}(2\kk)[q \g^{[2]}(\kk)]-q \g^{[2]}(\kk)  \mathsf{R}_{0}(2\kk) q \big\}dx, \label{A61-def}\\
	A^{[\geq 6]}_{2}(\kk, q)& =\int_{\T} \frac{2\kk}{2+\g(\kk)}\big\{ 2q \mathsf{R}_{0}(2\kk)[q \g^{[\geq 4]}(\kk)]-q \g^{[\geq 4]}(\kk)  \mathsf{R}_{0}(2\kk) q \big\}dx. \label{A62-def}
\end{align}
We have acceptable bounds on each of the terms \eqref{A61-def} and \eqref{A62-def}.

\begin{lemma} \label{LEM:Ageq6}
	Given $\kk\geq \tfrac{1}{2}$ and $q\in L^{2}(\T)$, it holds that
	\begin{align*}
		|A^{[\geq 6]}_{1}(\kk, q)|  & \les \kk^{-2}\|q\|_{L^{2}}^{4}(1+  \|q\|_{L^2}^{4}) \|\tfrac{q}{2\kk -\dd}\|_{L^\infty}\|\tfrac{q}{2\kk +\dd}\|_{L^{\infty}}, \\
		|A^{[\geq 6]}_{2}(\kk, q)| & \les \kk^{-3} \|q\|^6_{L^2} ( 1+ \|q\|^2_{L^2}).
	\end{align*}
\end{lemma}
\begin{proof}
	It follows from \eqref{gammap1} that
	\begin{align}
		\big\| \tfrac{1}{2+\g(\kk)}\big\|_{L^{\infty}} \leq 1. \label{2+g}
	\end{align}
	By \eqref{2+g}, \eqref{LinftyH1k}, and \eqref{gbd}
	\begin{align*}
		|A^{[\geq 6]}_{1}(\kk, q)| &\les \kk \|\g(\kk)\|_{L^{\infty}}\big\{  \| q \mathsf{R}_{0}(2\kk)[q \g^{[2]}(\kk)]\|_{L^{1}}+\|q \g^{[2]}(\kk)  \mathsf{R}_{0}(2\kk) q\|_{L^1} \big\} \\
		& \les (\|q\|_{L^2}^{2}+\kk^{-2}\|q\|_{L^2}^{6})  \big( \kk^{-2}\|q\|_{L^2}^{2} \|\tfrac{q}{2\kk -\dd}\|_{L^\infty}\|\tfrac{q}{2\kk +\dd}\|_{L^{\infty}}\big) \\
		& \les \kk^{-2}\|q\|_{L^{2}}^{4}(1+  \|q\|_{L^2}^{4}) \|\tfrac{q}{2\kk -\dd}\|_{L^\infty}\|\tfrac{q}{2\kk +\dd}\|_{L^{\infty}}
	\end{align*}
	Also, by \eqref{2+g}, \eqref{LinftyH1k}, \eqref{gbd2},
	\begin{align*}
		|A^{[\geq 6]}_{2}(\kk, q)| & \les \kk \|q \mathsf{R}_{0}(2\kk)[q \g^{[\geq 4]}(\kk)]\|_{L^1}+\kk \|q \g^{[\geq 4]}(\kk)  \mathsf{R}_{0}(2\kk) q \|_{L^1} \\
		& \les \kk^{-1} \|q\|_{L^2}^{2}\|\g^{[\geq 4]}(\kk)\|_{L^{\infty}} \\
		& \les \kk^{-3} \|q\|^6_{L^2} ( 1+ \|q\|^2_{L^2}).
	\end{align*}
	This completes the proof of the lemma.
\end{proof}

\section{Construction of the measures}\label{SEC:construct}

We start by elaborating on our choice of candidate invariant measure $\rho_{s,R}$ and base Gaussian measure $\mu_s$.
By defining the operator $w(-i\partial, \kk):= 4\kk^2 \mathsf{R_0}(2\kk) - \kk^2 \mathsf{R_0} (\kk)$, which has Fourier multiplier given by \eqref{wmult}, we can rewrite the quadratic-in-$q$ terms in $\mathcal{E}_s(q)$ from \eqref{Adecomposition} as follows
\begin{align*}
	\int_1^\infty \kk^{2s} \big[ A^{[2]}(\kk, q) - \tfrac12 A^{[2]}(\tfrac\kk2,q) \big] \, d\kk & = \tfrac12 \int_1^\infty \kk^{2s-1} \jb{q, w(-i\partial,\kk) q} \, d\kk = \tfrac12 \|\m_s(-i \partial) q\|^2_{L^2(\T)},
\end{align*}
	where the Fourier multiplier operator $\m_s(-i\partial)$ has corresponding multiplier given by
	 \eqref{wmult}. Now, we define the higher order in $q$ term $\mathcal{V}(\kk,q)$ in \eqref{Wkk} by
	\begin{align}
		\mathcal{V}(\kk,q) := A^{[\geq 4]}(\kk,q) -\tfrac{1}{2} A^{[\geq 4]}(\tfrac{\kk}{2},q), \label{EEdefn}
	\end{align}
	where $A^{[\geq 4]}(\kk, q)$ is as in \eqref{Adecomposition}.

\begin{lemma}\label{LEM:ms} Let $0\leq s<1$. Then, there exist constants $C>c>0$, independent of $\kk$, such that
\begin{align}
c\jb{\xi}^{2s} \leq \jb{\m_{s}(\xi)}^2\leq C\jb{\xi}^{2s} \label{mbd}
\end{align}
for all $\xi\in2\pi\Z$.
Hence, $\|\jb{ \m_{s}(-i\dd)}q\|_{L^{2}(\T)} \sim \| q\|_{H^{s}(\T)}$. Moreover, there exists a constant $C_{s}>0$, depending only on $s$, such that
\begin{align}
 \big| C_{s} |\xi|^{2s}-\mathfrak{m}_{s}(\xi)^2 \big| \les |\xi|^{-2} \label{Kakutanidiff}
\end{align}
for all $\xi\in 2\pi \Z$, $\xi\neq 0$.
\end{lemma}

\begin{proof}
Fix $\xi\in 2\pi \Z$ such that $\xi \neq 0$. From \eqref{wmult}, and a change of variables, we get
\begin{align*}
\mathfrak{m}_{s}(\xi)^2 & = 3 |\xi|^{2s} \int_{\frac{1}{|\xi|}}^{\infty} \frac{ \kk^{2s+1}}{(1+\kk^2)(1+4\kk^2)}d\kk  \sim  |\xi|^{2s}  \bigg( \int_{\frac{1}{|\xi|}}^{1} \kk^{2s+1}d\kk + \int_{1}^{\infty} \kk^{2s-3}d\kk  \bigg)   \sim |\xi|^{2s},
\end{align*}
as the second integral above converges as long as $s<1$. Then, to complete \eqref{mbd}, we just note that $1+|\xi|^{2s} \sim \jb{\xi}^{2s}$ as $s\geq 0$.
For \eqref{Kakutanidiff}, we choose
\begin{align}
C_{s} = 3\int_{0}^{\infty} \frac{ \kk^{2s+1}}{(1+\kk^2)(1+4\kk^2)} \, d\kk, \label{Cs}
\end{align}
which, as we noted above, converges for $-1< s<1$.
Then, we have
\begin{align*}
 \big| C_{s} |\xi|^{2s}-\mathfrak{m}_{s}(\xi)^2 \big| \les |\xi|^{2s} \int_{0}^{\frac{1}{|\xi|}} \kk^{1+2s} d\kk \les |\xi|^{-2},
\end{align*}
which yields \eqref{Kakutanidiff}.
\end{proof}

\subsection{Construction of the measures}
In this section, we construct the candidate invariant measures $\rho_{s, R}$ for every $\tfrac{1}{2}<s<1$ and $R>0$:
\begin{align}
d\rho_{s, R} =Z_{s,R}^{-1}\ind_{\{ \|q\|^2_{L^2} \leq R  \}}e^{-E(q)} d\mu_{s} = Z^{-1}_{s, R} F(q) \, d\mu_s, \label{Gibbs}
\end{align}
where
\begin{align}
E(q)= \int_{1}^{\infty} \kk^{2s}\mathcal{V}(\kk,q)d\kk \quad \text{and} \quad  F(q) = \ind_{\{ \|q\|_{L^2}^{2} \leq R  \}}e^{-E(q)}. \label{Edefn}
\end{align}
In order to make our computations rigorous, we need to introduce some truncations. Given $N\in \NB$, we let $\pi_{N}$ denote the projection onto frequencies $\{ \xi\in 2\pi\Z \, :\, |\xi|\leq 2\pi N\}$,~i.e.,
\begin{align*}
\pi_{N} f(x) = \sum_{|\xi|\leq 2\pi N} \ft f(n) e^{i\xi x}.
\end{align*}
We also let $\pi_{>N}:= \Id - \pi_N$ and for $N=\infty$ we define $\pi_{\infty} = \Id$.
Then, for $N\in \NB$, and $L>1$, we define
\begin{align}
F_{N,L}(q) & =  \ind_{\{ \|q\|_{L^2}^{2} \leq R  \}} e^{ -E_{L}(\pi_{N}q)} \qquad \text{where} \qquad
 E_{L}(\pi_{N} q)=  \int_{1}^{L} \kk^{2s}\mathcal{V}(\kk ,\pi_{N}q)d\kk. \label{EMJ}
\end{align}

Our goal is to prove the following result which says that the density $F(q)$ is integrable with respect to $\mu_{s}$, and hence the measures $\rho_{s,R}$ exist as probability measures.

\begin{proposition}\label{PROP:Finteq}
Fix $\tfrac{1}{2}<s<1$, $R>0$, and $1\leq p<\infty$. Then, there exists $C(p,R,s)>0$ such that
\begin{align*}
\| F(q)\|_{L^{p}(d\mu_{s})}\leq C(p,R,s)<\infty.
\end{align*}
In particular, the partition function $Z_{s,R}:=\E_{\mu_{s}}[F(q)]$ is finite.
\end{proposition}

Before proving Proposition~\ref{PROP:Finteq}, we establish the uniform bounds on the truncated densities $F_{N,L}(q)$ in $L^p(d\mu_s)$.

\begin{proposition}\label{PROP:FMJL}
	Fix $\tfrac{1}{2}<s<1$ and $R>0$. Then, there exists $C_0>0$ such that
	\begin{align*}
		\sup_{N\in \NB\cup \{\infty\}}\sup_{L>1} \| F_{N,L}(q)\|_{L^{p}(d\mu_{s})} \leq C_0<\infty,
	\end{align*}
	for any $1\leq p<\infty$.
\end{proposition}

In order to prove Proposition~\ref{PROP:FMJL}, we use a variational formula for which we need to first introduce some notations.
Let $W(t)$ be a cylindrical Brownian motion in $L^2(\T)$
\begin{align*}
W(t) = \sum_{\xi  \in  2\pi \Z} B_\xi(t) e^{i\xi x},
\end{align*}

\noi
where
$\{B_\xi\}_{\xi \in 2\pi \Z}$ is a sequence of mutually independent complex-valued\footnote{By convention, we normalize $B_\xi$ such that $\text{Var}(B_\xi(t)) = t$. In particular, $B_0$ is  a standard real-valued Brownian motion.} Brownian motions such that
$\cj{B_\xi}= B_{-\xi}$, $\xi \in 2\pi\Z$.
Then, define a centered Gaussian process $Y(t)$
by
\begin{align}
Y(t)
=  \jb{\mathfrak{m}_{s}(-i\dd)}^{-1}W(t).
\label{P2}
\end{align}

\noi
Note that
we have $\text{Law}(Y(1)) = \mu_{s}$,
where $\mu_{s}$ is the Gaussian measure in \eqref{gauss0}.
By setting  $Y_N = \pi_N Y $,
we have  $\text{Law}(Y_N(1)) = (\pi_N)_\#\mu_{s}$,
which is the pushforward of $\mu_{s}$ under $\pi_N$.
Also, let $\mathbb{H}_{a}$ denote the space of drifts,
which are progressively measurable processes
belonging to
$L^2([0,1]; L^2(\T^3))$, $\PP$-almost surely.
We now state the  Bou\'e-Dupuis variational formula \cite{BD, Ust};
in particular, see Theorem 7 in \cite{Ust}.

\begin{lemma}\label{LEM:var3}
Let $Y$ be as in \eqref{P2} and $N \in \NB \cup\{\infty\}$.
Suppose that  $G:C^\infty(\T) \to \R$
is measurable such that $\E\big[|G(\pi_N Y(1))|^p\big] < \infty$
and $\E\big[|e^{-G(\pi_N Y(1))}|^q \big] < \infty$, for some $1 < p, q < \infty$ with $\frac 1p + \frac 1q = 1$.
Then, we have
\begin{align*}
-\log \E\Big[e^{-G(\pi_N Y(1))}\Big]
= \inf_{\ta \in \mathbb H_a}
\E\bigg[ G(\pi_N Y(1) + \pi_N I(\ta)(1)) + \frac{1}{2} \int_0^1 \| \ta(t) \|_{L^2_x}^2 dt \bigg],
\end{align*}

\noi
where  $I(\ta)$ is  defined by
\begin{align*}
 I(\ta)(t) = \int_0^t \jb{\mathfrak{m}_{s}(-i\dd)}^{-1} \ta(t') dt'
\end{align*}

\noi
and the expectation $\E = \E_\PP$
is taken with respect to the underlying probability measure~$\PP$.

\end{lemma}

Before proceeding to the proof of Proposition~\ref{PROP:FMJL}, we state a lemma on the pathwise regularity bounds  of
$Y(1)$ and $I(\ta)(1)$.

\begin{lemma}  \label{LEM:Dr}

\textup{(i)}
Given any $1 \leq p <\infty$ and $\kk\geq 1$,
we have
\begin{align}
\begin{split}
\E
\Big[ & \|Y_N(1)\|_{W^{s-\frac{1}{2}-\eps,\infty}}^p
+ \kk^{p} \|(2\kk \pm \dd)^{-1} Y_{N}(1)\|_{L^{\infty}}^{p}
\Big]
\leq C_{\eps, p} <\infty,
\end{split}
\label{YLinfty}
\end{align}

\noi
 for any $\eps>0$ and uniformly in $N \in \NB\cup\{\infty\}$.

\smallskip

\noi
\textup{(ii)} For any $\ta \in \mathbb{H}_a$, we have
\begin{align}
\| I(\ta)(1) \|_{H^{s}}^2 \leq \int_0^1 \| \ta(t) \|_{L^2}^2dt.
\label{CM}
\end{align}
\end{lemma}
\begin{proof}
The proofs of the first bound in \eqref{YLinfty} and \eqref{CM} are well-known and can be found, for instance, in \cite[Lemma 4.7]{GOTW}. For the second bound in \eqref{YLinfty}, we argue in the usual way: given $\tfrac{1}{2}<s<1$, let $\eps>0$ be sufficiently small  such that $s-\tfrac{1}{2}-2\eps>0$, and consider $p\geq \tfrac{2}{\eps}$. Then, by Sobolev embedding, Minkowski's inequality, the representation \eqref{randomfourier}, and \eqref{mbd}, we have
\begin{align*}
\E
\Big[  \|(2\kk \pm \dd)^{-1} Y_{N}(1)\|_{L^{\infty}_x}^{p}\Big]^{\frac{1}{p}} & \les \E\big[ \| (2\kk \pm \dd)^{-1} Y_{N}(1)\|_{W^{2\eps,\frac1\eps}_x}^{p}]^{\frac{1}{p}} \\
& \les \bigg\| \E\bigg[ \bigg( \sum_{|\xi|\leq 2\pi N} \frac{\jb{\xi}^{2\eps}g_{\xi}(\o)}{\jb{\mathfrak{m}_{s}(\xi)}}(2\kk\pm i\xi)^{-1}e^{i\xi x}\bigg)^{p}\bigg]^{\frac{1}{p}}  \bigg\|_{L^{\frac1\eps}_x(\T)} \\
& \les \bigg( \sum_{|\xi|\leq 2\pi N} \frac{\jb{\xi}^{4\eps}}{ \jb{\mathfrak{m}_{s}(\xi)}^{2}(4\kk^2+\xi^2)}\bigg)^{\frac{1}{2}}  \les \kk^{-1}.
\end{align*}
The result for $1\leq p<\tfrac{2}{\eps}$ then follows by H\"older's inequality.
\end{proof}

\begin{remark}\rm  \label{RMK:Linfty}
As $Y(1)$ has the same law as a sample of the Gaussian measure $\mu_s$, we note that \eqref{YLinfty} also holds with the expectation over $\mathbb{P}$ replaced by the expectation over $\mu_s$, and $Y_{N}(1)$ replaced by $\pi_{N} q$.
\end{remark}

\begin{proof}[Proof of  Proposition~\ref{PROP:FMJL}]
Fix $1\leq p <\infty$, $J\in \NB$, $L>1$, and $N\in \NB\cup\{\infty\}$. Since
\begin{align*}
\E_{\mu_{s}}\Big[ \ind_{\{ \|q\|^2_{L^{2}}\leq R\}} e^{p\min(-E_{L}(\pi_N q),J)}\Big] \leq \E_{\mu_{s}}\Big[ e^{p\min(-E_{L}(\pi_N q),J)\ind_{\{ \|q\|^2_{L^{2}}\leq R\}} }\Big] ,
\end{align*}
it suffices to prove
\begin{align}
\sup_{N\in\NB\cup\{\infty\}}\sup_{L>1} \, \sup_{J\in \NB}\E_{\mu_{s}}\Big[ e^{p\min(-E_{L}(\pi_N q),J)\ind_{\{ \|q\|^2_{L^{2}}\leq R\}} }\Big]  \leq C(p,R)<\infty. \label{BDbd}
\end{align}
Thus, by the variational formula (Lemma~\ref{LEM:var3}) and Lemma~\ref{LEM:Dr}, we have
\begin{align}
&\log \E_{\mu_{s}}\Big[ e^{p\min(-E_{L}(\pi_N q),J)\ind_{\{ \|q\|^2_{L^{2}}\leq R\}} }\Big]  \notag \\
& = \log \E\Big[ e^{p\min(-E_{L}(\pi_{N} Y(1)),J)\ind_{\{ \|Y(1)\|^2_{L^{2}}\leq R\}} }\Big]\notag  \\
& = \sup_{\ta\in \mathbb{H}_{a}}\E\bigg[ p\min(-E_{L}(\pi_N Y(1)+\pi_NI(\ta)(1)),J)\ind_{\{ \|Y(1)+I(\ta)(1)\|^2_{L^{2}}\leq R\}} -\frac{1}{2}\int_{0}^{1} \|\ta(t)\|_{L^2}^2 dt\bigg]\notag \\
& \leq \sup_{\ta\in \mathbb{H}_{a}}\E\bigg[ -pE_{L}(\pi_{N} Y(1)+\pi_N I(\ta)(1))\ind_{\{ \|Y(1)+I(\ta)(1)\|^2_{L^{2}}\leq R\}} -\frac{1}{2} \|I(\ta)(1)\|_{H^{s}}^{2} \bigg], \label{varformula}
\end{align}
where in the first line the expectation is with respect to $\mu_{s}$, while in the remaining lines, the expectation is with respect to the underlying probability measure $\mathbb{P}$. Moreover, the upper bound in \eqref{varformula} is uniform in $J\in \NB$. We remark that the only purpose of $J$ is to ensure that we satisfy the integrability assumptions in Lemma~\ref{LEM:var3}.

For simplicity, let $Y:=Y(1)$, $Y_{N} :=\pi_{N} Y$, $\Dr:=I(\ta)(1)$, and $\Dr_{N}:=\pi_{N} \Dr$.
Thus, \eqref{BDbd} follows once we prove a uniform in $N$ and $L$ bound on \eqref{varformula}. We recall from \eqref{EMJ} that
\begin{align*}
E_{L}(Y_N+\Dr_N) =\int_{1}^{L} \kk^{2s} \mathcal{V}(\kk ,Y_N+\Dr_N) d\kk.
\end{align*}
Using \eqref{EEdefn}, \eqref{A4}, and \eqref{Ageq6}, we further expand this as
\begin{align}
\begin{split}
E_{L}(Y_N+\Dr_N)
& = \int_{1}^{L} \kk^{2s} \big\{ A^{[4]}(\kk , Y_N+\Dr_N)-\tfrac{1}{2} A^{[4]}(\tfrac{\kk}{2} , Y_N+\Dr_N)\big\} d\kk\\
&\hphantom{X} +   \int_{1}^{L} \kk^{2s}\big\{ A^{[\geq 6]}(\kk, Y_N+\Dr_N)-\tfrac{1}{2} A^{[\geq 6]}(\tfrac{\kk}{2}, Y_N+\Dr_N)\big\} d\kk.
\end{split} \label{EMJdecomp}
\end{align}

Returning to \eqref{varformula}, we have from \eqref{EMJdecomp}, \eqref{A4general}, Lemma~\ref{LEM:interp}, Tonnelli's theorem, Lemma~\ref{LEM:Ageq6}, \eqref{YLinfty}, H\"older's and Young's inequalities,
\begin{align*}
	&\E\bigg[ -pE_{L}(Y_N+\Dr_N)\ind_{\{ \|Y_N+\Dr_N\|^2_{L^{2}}\leq R\}} -\frac{1}{2} \|\Dr\|_{H^{s}}^{2} \bigg] \\
	& \leq \E\bigg[  \bigg\{  C_{p} \|Y_N+\Dr_N\|_{L^2}^{2}  \int_{1}^{L} \kk^{2s-1} \bigg(  \| \tfrac{Y_N}{2\kk -\dd}\|_{L^\infty}\| \tfrac{Y_N}{2\kk +\dd}\|_{L^{\infty}} \\
	& \hphantom{XXX}+ \| \tfrac{Y_N}{2\kk -\dd}\|_{L^\infty}\| \tfrac{\Dr_N}{2\kk +\dd}\|_{L^{\infty}}+\| \tfrac{\Dr_N}{2\kk-\dd}\|_{L^\infty}\| \tfrac{Y_N}{2\kk +\dd}\|_{L^{\infty}}+\| \tfrac{\Dr_N}{2\kk -\dd}\|_{L^\infty}\| \tfrac{\Dr_N}{2\kk +\dd}\|_{L^{\infty}} \bigg) d\kk\\
	&  \hphantom{XXX}+ C_{p}(1+\| Y_{N} +\Dr_N\|_{L^2}^{8})\int_{1}^{L} \kk^{2s-2} \|\tfrac{Y_N+\Dr_N}{2\kk  -\dd}\|_{L^\infty}\|\tfrac{Y_N+\Dr_N}{2\kk  +\dd}\|_{L^{\infty}} d\kk \\
	& \hphantom{XXX}+ C_{p}(1+\|Y_N+\Dr_N\|_{L^2}^{8}) \bigg\} \ind_{\{ \|Y_N+\Dr_N\|^2_{L^{2}}\leq R\}}  -\tfrac{1}{2}\|\Dr \|_{H^{s}}^{2}\bigg] \\
	& \leq C_p R \int_{1}^{L} \kk^{2s-1}\bigg\{ \kk^{-2} +\kk^{-1}\E\bigg[  \|\tfrac{Y_N}{2\kk  \pm \partial}\|_{L^\infty} \|\Dr_N\|_{H^s} +  \kk^{-\theta} (\|Y_N\|_{L^2}^2+R)^{1-\theta} \|\Dr_N\|_{H^s}^{2\theta}\bigg] \bigg\} d\kk\\
	&  \hphantom{XX}+C_p(1+R^5) \int_{1}^{L} \kk^{2s-3}  d\kk + C_p (1+R^4)-\E\big[ \tfrac{1}{2} \|\Dr\|_{H^{s}}^{2} \big]
	\\
	& \leq C_p (1+R^4) \int_{1}^{L} \kk^{2s-2} \E\Big[ \big\| \tfrac{Y_N}{2\kk  \pm \partial } \big\|_{L^\infty} \|\Dr_N\|_{H^s} + \kk^{-\theta} (\|Y_N\|_{L^2}^2+R)^{1-\theta} \|\Dr_N\|_{H^s}^{2\theta}\Big]d\kk\\
	& \phantom{XXX} + C_p(1+R^4) - \E[\tfrac12 \|\Dr \|^2_{H^s}] \\
	& \leq C_p(1+R^4) + C_p (1+R^4) \int_{1}^{L} \kk^{2s-2}  \E\big[ \big\| \tfrac{Y_N}{2\kk \pm \partial} \big\|_{L^2}^{2\wt{p}} \big]^{\frac{1}{2\wt{p}}} \E \big[ \|\Dr_N\|_{H^s}^{1+} \big]^{\frac{1}{1+}}d\kk \\
	& \phantom{XXX} + C_p (1+R^4) \int_{1}^{L} \kk^{2s-2-\theta} \E\Big[ C_\dl (\|Y_N\|^2_{L^2}+R)^{(1+\theta)C} + \dl \|\Dr_N\|_{H^s}^2 \Big]d\kk - \E[ \tfrac12 \|\Dr\|_{H^s}^2]\\
	& \leq C_p (1+R^4) + C_p (1+R^4)  \E\big[ \|\Dr_N \|_{H^s}^{1+} \big]^{\frac{1}{1+}}  - \E[\tfrac14 \|\Dr\|_{H^s}^2]\\
	& \leq C_{p,R}+ \E[ \tfrac{1}{16} \|\Dr_N\|_{H^{s}}^{2}]-\E\big[ \tfrac{1}{4} \|\Dr\|_{H^{s}}^{2} \big]  \\
	& \leq C_{p,R}
\end{align*}
uniformly in $\Dr$, and where we chose $0<\ta<1$ such that $2s-1<\theta<1$, $\dl>0$ sufficiently small, and $2\leq \wt{p}<\infty$ sufficiently large.
This shows that
\begin{align*}
\log \E_{\mu_{s}}\Big[ \exp\Big({p\min(-E_{L}(\pi_N q),J)\ind_{\{ \|q\|^2_{L^{2}}\leq R\}} } \Big)\Big]  \leq C_{p,R}
\end{align*}
uniformly in $ J\in \NB$, $L>1$, and $N\in \NB \cup \{\infty\}$. Now, by the monotone convergence theorem, we have
\begin{align*}
\E_{\mu_{s}}\Big[ \exp\Big(-pE_{L}(\pi_N q)\ind_{\{ \|q\|^2_{L^{2}}\leq R\}} \Big)\Big]  \leq C_{p,R}
\end{align*}
uniformly in $N\in \NB\cup \{\infty\}$ and $L>1$, which completes the proof of Proposition~\ref{PROP:FMJL}.
\end{proof}

With the uniform bound of Proposition~\ref{PROP:FMJL} in place, we now aim to take limits, first as $L\to \infty$ and then as $N\to \infty$.
For the first limit, we need the following convergence result.
\begin{lemma}\label{LEM:EJM}
Let $\tfrac{1}{2}<s<1$. Then, there exists $\ta>0$ such that
\begin{align}
| E_{L_1}(\pi_N q)-E_{L_2}(\pi_N q)|  \les (L_1 \wedge L_2)^{-\ta} (1+\|\pi_{N} q\|_{L^{\infty}}^{C})\label{EJcauchy}
\end{align}
where the implicit constant is uniform in $N\in \NB \cup \{\infty\}$ and $L_1,L_2>1$, and for some $C>0$. Moreover, $E$ in \eqref{Edefn} is well-defined as an absolutely convergent improper Riemann integral $\mu_{s}$-a.s., $E$ belongs to $L^p(d\mu_{s})$, and it satisfies
\begin{align*}
 \sup_{N\in \NB \cup \{\infty\}} \|E(\pi_{N}q)- E_{L}(\pi_{N}q)\|_{L^{p}(d\mu_{s})} \les_{p} L^{-\ta}.
\end{align*}
In particular, $E_{L}(\pi_{N} q)\to E(\pi_{N} q)$ in $L^{p}(d\mu_{s})$, uniformly in $N\in \NB \cup \{\infty\}$, and $\mu_{s}$-a.s.
\end{lemma}

\begin{proof}
We first prove \eqref{EJcauchy}.
By symmetry, we suppose that $L_1 >L_2$. Recall by \eqref{EEdefn}, \eqref{A4}, and \eqref{Ageq6} that we wrote
\begin{align*}
\mathcal{V}(\kk ,q)& = A^{[4]}(\kk ,q) - \tfrac{1}{2}A^{[4]}(\tfrac{\kk}{2} ,q)  +\sum_{j=1}^{2} \big( A_{j}^{[\geq 6]}(\kk ,q) - \tfrac{1}{2}A_{j}^{[\geq 6]}(\tfrac{\kk}{2} ,q) \big).
\end{align*}
It follows from \eqref{A4}, Cauchy-Schwarz and \eqref{LinftyH1k} that
\begin{align}
| A^{[4]}(\kk , q)| \les \kk^{-3}\|q\|_{L^{\infty}}^{4}.  \label{A4EL}
\end{align}
Similarly, from \eqref{A61-def}, \eqref{A62-def}, \eqref{2+g}, Cauchy-Schwarz, \eqref{gbd}, \eqref{gbd2}, \eqref{G2}, and \eqref{LinftyH1k},
\begin{align}
| A^{[\geq 6]}_{1}(\kk , q)| + | A^{[\geq 6]}_{2}(\kk , q)|\les \kk^{-\frac{7}{2}}  \|q\|_{L^{\infty}}^{5}(1+\|q\|_{L^{\infty}}^{5}). \label{A6EL}
\end{align}
Thus, \eqref{A4EL} and \eqref{A6EL} imply
\begin{align}
|\mathcal{V}(\kk,q)|\les  \kk^{-3} \|q\|_{L^{\infty}}^{5}(1+\|q\|_{L^{\infty}}^{5}). \label{VEL}
\end{align}
Notice that \eqref{VEL} holds for $q$ replaced by $\pi_{N}q$ for any $N\in \NB$, at which point, \eqref{EJcauchy} follows from \eqref{EMJ}.
Now, \eqref{EJcauchy} yields the existence of the limit
\begin{align*}
E(q):=\lim_{L\to \infty} E_{L}(q) = \int_{1}^{\infty} \kk^{2s} \mathcal{V}(\kk,q)d\kk,
\end{align*}
$\mu_{s}$-a.s. in view of Remark~\ref{RMK:Linfty}.
Moreover, $E(q)\in L^{p}(d\mu_{s})$ for any $1\leq p<\infty$ due to \eqref{VEL}.

It follows from \eqref{EJcauchy}, that for each fixed $N\in \NB$, there exists $\wt{E}^{(N)}\in L^{p}(d\mu_{s})$, such that
\begin{align}
\lim_{L\to \infty}  E_{L}(\pi_N q) =:  \wt{E}^{(N)}(q), \label{diffEJM1}
\end{align}
where the convergence is uniform in $N$.
We now show that $\mu_{s}$-a.s.
\begin{align}
E\circ \pi_N=\wt{E}^{(N)} \quad \text{for all} \,\, N\in \NB. \label{EequalswtE}
\end{align}
Note that $E$ in \eqref{Edefn} satisfies
\begin{align*}
\| E_{L}( q) - E(q)\|_{L^p(d\mu_{s})} \les_{p} L^{-\ta}.
\end{align*}
 Fix $N\in \NB$. By the triangle inequality and \eqref{diffEJM1}, we have
\begin{align}
\| E(\pi_N q) - \wt{E}^{(N)}(q)\|_{L^p(d\mu_{s})}  & \leq \|E(\pi_N q)-E_{J}(\pi_N q)\|_{L^p(d\mu_{s})} +CL^{-\ta} \label{EwtEdiff}.
\end{align}
Since
\begin{align*}
E(\pi_N q)-E_{L}(\pi_N q) =\int_{L}^{\infty} \kk^{2s} \mathcal{V}(\kk , \pi_{N}q)d\kk,
\end{align*}
it follows from \eqref{VEL} with $\pi_{N}q$ instead of $q$ that
\begin{align*}
\sup_{N\in \NB}  \int_{L}^{\infty} \kk^{2s}| \mathcal{V}(\kk , \pi_{N}q) | d\kk \les L^{-\ta}\|\pi_{N} q\|_{L^{\infty}}^{5}(1+\|\pi_{N}q\|_{L^{\infty}}^{5}).
\end{align*} Using this bound in \eqref{EwtEdiff} and taking $L\to \infty$ establishes \eqref{EequalswtE}, as well as the $L^p(d\mu_s)$ and almost sure convergence of $E_{L}(\pi_{N} q)$ to $E(\pi_{N} q)$, uniformly in $N\in \NB$, as $L\to \infty$. \qedhere

\end{proof}

This convergence now propagates to the densities $F_{N,L}(q)$.

\begin{lemma}\label{LEM:FMJ}
Let $\tfrac{1}{2}<s<1$ and $R>0$. Then,
\begin{align}
\lim_{L\to \infty} \|F_{N,L}(q) - F_{N}(q)\|_{L^{p}(d\mu_{s})}=0,
\qquad \text{where} \qquad
F_{N}(q) := \ind_{\{ \|q\|_{L^2}^{2}\leq R\}} e^{-E(\pi_{N}q)} \label{FM}
\end{align}
and $E$ given in Lemma~\ref{LEM:EJM}. Moreover, we have
\begin{align}
\sup_{N\in \NB} \| F_{N}(q)\|_{L^{p}(d\mu_{s})}\leq C<\infty. \label{FMunifbd}
\end{align}
\end{lemma}
\begin{proof}
The proof is a standard consequence of the uniform bound in Proposition~\ref{PROP:FMJL} and the almost sure convergence of $E_{L}(\pi_N q)$ to $E(\pi_N q)$ as $L\to \infty$, which was proven in Lemma~\ref{LEM:EJM}. See for example \cite[Proof of Proposition 6.2]{OTz}.
The uniform bound \eqref{FMunifbd} now follows from the convergence in \eqref{FM} and the uniform bound from Proposition~\ref{PROP:FMJL}.
\end{proof}

The final step is taking $N\to \infty$, which requires an additional convergence result.

\begin{lemma}\label{LEM:EMconv}
Let $\tfrac{1}{2}<s<1$ and $R>0$. Then, there exists $\ta'>0$ such that
\begin{align}
\|E(\pi_{N_1}q)-E(\pi_{N_2}q)\|_{L^{p}(d\mu_{s})} \les_{p, s}(N_1\wedge N_2)^{-\ta'} \label{EM1M2}
\end{align}
for all $1\leq p<\infty$, and $N_1,N_2\in \NB$. In particular, $E$ defined in \eqref{Edefn} satisfies
\begin{align}
\| E(\pi_{N}q)-E(q)\|_{L^{p}(d\mu_{s})} \les_{p, s} N^{-\ta'}. \label{EpiMbd}
\end{align}
\end{lemma}

Assuming the validity of Lemma~\ref{LEM:EMconv}, we now complete the construction of the measures $\rho_{s,R}$. Indeed, arguing similarly as in the proof of Lemma~\ref{LEM:FMJ}, \eqref{EpiMbd} implies that $E(\pi_{N}q)$ converges to $E(q)$, almost surely with respect to $\mu_{s}$. Combining this with the uniform bound in \eqref{FMunifbd}, shows that
\begin{align}
\lim_{N\to \infty} \| F_{N}(q) -F(q)\|_{L^p(d\mu_{s})} =0 \label{LpconvFM}
\end{align}
for any $1\leq p<\infty$. Thus, the partition function
$Z_{s,R}:=\E_{\mu_{s}}[ F(q)]$
is finite and the measure
\begin{align*}
d\mu_{s,R} : = Z_{s,R}^{-1}F(q)d\mu_{s}
\end{align*}
exists as a probability measure which is absolutely continuous with respect to the Gaussian measure $\mu_{s}$. Moreover, we have equivalence of $\rho_{s,R}$ with the measure in \eqref{gauss-rest} as we did not make use of the sign out the front of $E(q)$ in \eqref{Edefn}.

It remains to prove Lemma~\ref{LEM:EMconv}.

\begin{proof}[Proof of Lemma~\ref{LEM:EMconv}]

Given $q_1,q_2 \in L^{\infty}(\T)$ and $\kk\geq 1$, we first consider the differences
 \begin{align*}
&\mathcal{V}(\kk,q_1)-\mathcal{V}(\kk,q_2) \notag \\
&  =A^{[4]}(\kk,q_1)-A^{[4]}(\kk, q_2)- \tfrac{1}{2}(A^{[4]}(\tfrac{\kk}{2},q_1)-A^{[4]}(\tfrac{\kk}{2}, q_2)  )  \\
& \hphantom{X} + \sum_{j=1}^{2} \big\{ A^{[\geq 6]}_{j}(\kk,q_1)-\tfrac{1}{2}A^{[\geq 6]}_{j}(\kk, q_2) \big\}-\frac{1}{2}\sum_{j=1}^{2} \big\{ A^{[\geq 6]}_{j}(\tfrac{\kk}{2},q_1)-\tfrac{1}{2}A^{[\geq 6]}_{j}(\tfrac{\kk}{2}, q_2) \big\}.
\end{align*}
It follows from \eqref{A4} and Cauchy-Schwarz that we have
\begin{align}
\begin{split}
|A^{[4]}(\kk,q_1)-A^{[4]}(\kk,q_2)| \les \kk^{-3} \|q_1-q_2\|_{L^\infty}(\|q_1\|_{L^{\infty}}^{3}+\|q_2\|_{L^{\infty}}^{3}).
\end{split} \label{A4diff}
\end{align}
Next, we have
\begin{multline*}
A^{[\geq 6]}_{1}(\kk,q_1)-A^{[\geq 6]}_{1}(\kk,q_2) = \int_{\T} \tfrac{\kk( \g_{2}(\kk)-\g_1(\kk) )}{(2+\g_{1}(\kk))(2+\g_{1}(\kk))}\g_{1}(\kk) Q(\kk,q_1)dx \\+  \int_{\T} \tfrac{\kk( \g_{2}(\kk)-\g_1(\kk) )}{2+\g_{1}(\kk)}\g_{1}(\kk) Q(\kk,q_1)dx
+\int_{\T} \tfrac{\kk \g_{2}(\kk)}{2+\g_{2}(\kk)}\big\{ Q(\kk,q_1)-Q(\kk,q_2)\big\}dx,
\end{multline*}
where we defined $\g_{j}(\kk):=\g(\kk,q_j)$, $j=1,2$, and
\begin{align*}
 Q(\kk,q):=2q \mathsf{R}_{0}(2\kk)[q \g^{[2]}(\kk)]-q \g^{[2]}(\kk)  \mathsf{R}_{0}(2\kk) q.
\end{align*}
Now, it follows from \eqref{LinftyH1k} and \eqref{G2}, that
\begin{align}
\|Q(\kk,q)\|_{L^2}\les \kk^{-3}\|q\|_{L^{\infty}}^{2} ( 1+ \|q\|_{L^\infty}^2). \label{Qbound}
\end{align}

\noi
Moreover, by a similar computation to \eqref{A4diff}, we get
\begin{align}
\|Q(\kk,q_1)-Q(\kk,q_2)\|_{L^1}\les \kk^{-4}(1+\|q_1\|_{L^{\infty}}^{4}+\|q_2\|_{L^{\infty}}^{4})\|q_1-q_2\|_{L^\infty}.
 \label{Qdiff}
\end{align}
Thus, using \eqref{Qbound}, \eqref{Qdiff}, \eqref{gbd}, \eqref{g lip},  and \eqref{2+g}, we have
\begin{align}
& |A^{[\geq 6]}_{1}(\kk,q_1)-A^{[\geq 6]}_{1}(\kk,q_2)| \les \kk^{-4}\|q_1-q_2\|_{L^\infty} (1+\|q_1\|_{L^{\infty}}^{17}+\|q_{2}\|_{L^{\infty}}^{17}). \label{A61diff}
\end{align}
Moreover, from \eqref{2+g}, \eqref{gbd}, \eqref{gbd2}, and \eqref{g lip4}, we have
\begin{align}
&|A^{[\geq 6]}_{2}(\kk,q_1)-A^{[\geq 6]}_{2}(\kk,q_2)| \les \kk^{-\frac72}\|q_1-q_2\|_{L^\infty} (\|q_1\|_{L^{\infty}}^{15}+\|q_{2}\|_{L^{\infty}}^{15}). \label{A62diff}
\end{align}

Now, recalling the definition \eqref{Edefn}, we have from \eqref{A4diff}, \eqref{A61diff}, \eqref{A62diff}, Cauchy-Schwarz, \eqref{YLinfty}, \eqref{randomfourier}, and \eqref{mbd}, that
\begin{align*}
 \|E(\pi_{N_1}q)-E(\pi_{N_2}q)\|_{L^{p}(d\mu_{s})}   &\les    \big\| \|\pi_{N_1}q-\pi_{N_2}q\|_{L^{\infty}} (1+\|q\|_{L^{\infty}}^{C}  )\big\|_{L^{p}(d\mu_{s})} \\
& \les \big\| \|\pi_{N_1}q-\pi_{N_2}q\|_{L^{\infty}}  \big\|_{L^{2p}(d\mu_{s})} \\
& \les \bigg( \sum_{|\xi| > 2\pi(N_1 \wedge N_2)} \frac{1}{\jb{\xi}^{2s-\dl}} \bigg)^\frac12 \\
& \les (N_1\wedge N_2)^{-s+\frac12 + \frac\dl2},
\end{align*}
where $C\in\NB$ is a large enough power,
since $s>\tfrac{1}{2}$, and where $\dl>0$ is sufficiently small. This proves \eqref{EM1M2}, and \eqref{EpiMbd} follows by taking $N_1 \wedge N_2 \to \infty$ in \eqref{EM1M2}.
\end{proof}

\subsection{Equivalence of base Gaussian measures}\label{SEC:equiv}

We now show that the Gaussian measures $\mu_{s}$ in \eqref{gauss0} on $L^2(\T)$ are equivalent to the Gaussian measures $\wt{\mu}_{s}$ in \eqref{mutilde}.
 Rigorously, we understand $\wt{\mu}_{s}$ as the law of the random Fourier series
\begin{align*}
\o\in \O \mapsto \wt{q}(x;\o) =\frac{1}{\sqrt{C_{s}} } \sum_{\xi\in 2\pi\Z} \frac{g_{\xi}(\o)}{\jb{\xi}^{s}}e^{i\xi x},
\end{align*}
where $(\O,\mathcal{F},\mathbb{P})$ is the same underlying probability space as for \eqref{randomfourier} and  $C_{s}>0$ is the constant from \eqref{Kakutanidiff}. The Gaussian measures $\wt{\mu}_s$ are more commonly used in the study of nonlinear dispersive PDE with random initial data or stochastic forcing.
If one wishes for there to not be the constant $(\sqrt{C_{s}})^{-1}$ in the expansion \eqref{randomfourier}, then one instead needs to insert it appropriately into \eqref{gauss0}. The effect of this would be to alter the inverse temperature for the invariant measures \eqref{GibbsEmeas}, which has no affect on their construction.

\begin{proposition} \label{PROP:equiv}
For every $\tfrac{1}{2}<s<1$, the Gaussian measure $\mu_{s}$ in \eqref{gauss0} is equivalent to the Gaussian measure $\wt{\mu}_{s}$ in \eqref{mutilde}.
\end{proposition}

The poof of Proposition~\ref{PROP:equiv} employs Kakutani's theorem \cite{Kak}, known as the Feldman-H\'{a}jek theorem in the Gaussian setting~\cite{Feldman, Hajek}. See, for instance, \cite[Proposition B.1]{BTz} for a use of this theorem in the context of nonlinear dispersive PDEs, as well as \cite{LOZ} for further references.
We will use the following statement, tailored for our context, whose proof can be found in \cite[Lemma 3.2]{LOZ}.

\begin{lemma}\label{LEM:Kak}
Let $\{A_{\xi}\}_{\xi\in 2\pi \Z}$ and $\{B_{\xi}\}_{\xi\in 2\pi\Z}$ be two sequences of independent, real-valued mean-zero Gaussian random variables with $\E[A_{\xi}^2]=a(\xi)>0$ and $\E[B_{\xi}^{2}]=b(\xi)>0$ for all $\xi\in 2\pi \Z$. Then, the laws of $\{A_{\xi}\}_{\xi\in 2\pi \Z}$ and $\{B_{\xi}\}_{\xi\in 2\pi\Z}$ are equivalent if and only if
\begin{align}
\sum_{\xi\in 2\pi \Z} \bigg( \frac{a(\xi)}{b(\xi)}-1\bigg)^{2}<\infty. \label{equivsum}
\end{align}
If they are not equivalent, then they are singular.
\end{lemma}

\begin{proof}[Proof of Proposition~\ref{PROP:equiv}]
We rewrite \eqref{randomfourier} in the following way:
\begin{align*}
q(\o) = g_{0}(\o) + \sum_{\xi\in 2\pi \NB} \bigg(  \frac{\text{Re} \, g_{\xi}(\o)}{ \jb{\mathfrak{m}_{s}(\xi)}} \cos(\xi x) - \frac{\text{Im} \, g_{\xi}(\o)}{ \jb{\mathfrak{m}_{s}(\xi)}} \sin(\xi x) \bigg).
\end{align*}
Then, for each $\xi\in 2\pi \NB$, we set $A_0=1$, $B_0 = (\sqrt{C_s})^{-1}$, and
\begin{align*}
A_{\xi} = \frac{\text{Re} \, g_{\xi}}{ \jb{\mathfrak{m}_{s}(\xi)}}, \quad A_{-\xi} =  - \frac{\text{Im} \, g_{\xi}}{ \jb{\mathfrak{m}_{s}(\xi)}}, \quad B_{\xi} = \frac{1}{\sqrt{C_s}}\frac{\text{Re} \, g_{\xi}}{ \jb{\xi}^{s}}, \quad B_{-\xi} =  -  \frac{1}{\sqrt{C_s}}\frac{\text{Im} \, g_{\xi}}{ \jb{\xi}^{s}},
\end{align*}

\noi
Then,
\begin{align*}
a(\xi) = \E[A_{\xi}^{2}]= \frac{1}{   \jb{\mathfrak{m}_s (\xi)}^2}, \quad \text{and} \quad b(\xi)=\E[B_{\xi}^2] =\frac{1}{C_{s}\jb{\xi}^{2s}}.
\end{align*}
Now, it remains to check that \eqref{equivsum} holds for these sequences $\{ a(\xi)\}_{\xi\in 2\pi \Z}$ and $\{b(\xi)\}_{\xi \in 2\pi \Z}$. Fix $\xi \in 2\pi \Z$, $\xi\neq 0$. With the function $f(x)=x^{s}$ for $x>0$, the mean value theorem implies
\begin{align*}
|\jb{\xi}^{2s} -|\xi|^{2s}| = |f(1+|\xi|^{2})-f(|\xi|^{2})| \les |\xi|^{2(s-1)}.
\end{align*}
Thus, by \eqref{Kakutanidiff} and since $s<1$, we have
\begin{align}
|C_{s} \jb{\xi}^{2s}-\jb{\mathfrak{m}_{s}(\xi)}^{2}| &  \les 1+|\xi|^{-2}+|\xi|^{2(s-1)} \les 1. \label{diffbdKak}
\end{align}
Then, by \eqref{mbd}, \eqref{diffbdKak}, and that $s>\tfrac{1}{4}$, we have
\begin{align*}
\sum_{\xi \in 2\pi \Z} \bigg(\frac{a(\xi)}{b(\xi)}-1\bigg)^{2} \les 1+ \sum_{\substack{\xi\in 2\pi \Z \\ \xi \neq 0}} \frac{(C_{s} \jb{\xi}^{2s}-\jb{\mathfrak{m}_{s}(\xi)}^{2})^2}{\jb{\mathfrak{m}_{s}(\xi)}^{4}} \les 1+\sum_{\substack{\xi\in 2\pi \Z \\ \xi \neq 0}} \frac{1}{\jb{\xi}^{4s}} <\infty.
\end{align*}
Therefore, an application of Lemma~\ref{LEM:Kak} completes the proof.
\end{proof}

\section{Invariance} \label{SEC:invariance}

In order to establish the invariance of the measures $\rho_{s,R}$ in \eqref{Gibbs} under the $H^{\text{mKdV}}$-flow, we first establish their invariance under the $H_\vk^\text{mKdV}$-flow, by considering suitably ``truncated'' dynamics induced by $H_{\vk, N}^{\text{mKdV}}$. We then extend this result to invariance for mKdV due to the good approximation between the $H_\vk^\text{mKdV}$ and $H^\text{mKdV}$ flows (Proposition~\ref{P:HK well}). Note that we exclusively use $\vk$ to index the approximating flows $H_\vk^\text{mKdV}$ to distinguish it from the spectral parameter and integration  variable $\kk$ in \eqref{Wkk}.

Let $N\in\NB$. We define new Gaussian measures as $\mu_{s, N} = \mu_{s,N} \otimes \mu^\perp_{s, N}$, where
\begin{align*}
	d \mu_{s,N} &= Z^{-1}_N \exp\bigg( - \frac12 \sum_{|\xi| \leq 2\pi N } |\mathfrak{m}_s(\xi )|^2 |\ft{q}(\xi)|^2  - \frac12\sum_{|\xi| \leq 2\pi N} |\ft{q}(\xi)|^2\bigg) \prod_{|\xi| \leq 2\pi N} d\ft{q}(\xi) , \\
	d \mu^\perp_{s,N} &= \wt{Z}^{-1}_N \exp\bigg( - \frac12  \sum_{|\xi| > 2 \pi N } |\mathfrak{m}_s(\xi )|^2 |\ft{q}(\xi)|^2  -\frac12 \sum_{|\xi| >2 \pi N} |\ft{q}(\xi)|^2\bigg) \prod_{|\xi| > 2 \pi N} d\ft{q}(\xi) ,
\end{align*}

\noi
and define the finite dimensional measure $\wt{\rho}_{s,N}$ with density
\begin{equation*}
	d \wt{\rho}_{s,N}  = Z^{-1}_{s,N} F(\pi_N q) \, d  \mu_{s, N}(q).
\end{equation*}
Then, the ``truncated'' measures $\rho_{s,N}$ are given by $\rho_{s, N} := \wt{\rho}_{s,N} \otimes \mu^\perp_{s,N}$. Note that $\mu_{s} = \mu_{s, N} \otimes \mu^\perp_{s,N}$ for the Gaussian measures in \eqref{gauss0}.

For the truncated Hamiltonian
${H_{\vk, N}^{\text{mKdV}}(q) : = H_{\vk}^{\text{mKdV}}(\pi_N q)}$, \eqref{Hvk}~and~\eqref{derivA}~imply
\begin{align}
	\tfrac{\dl H_{\vk,N}^{\text{mKdV}}}{ \dl q}
	= \pi_N \big( 4\vk^2 \pi_N q - 4\vk^3 g_{-}(\vk, \pi_N q) \big).
	\label{Hkk-M-diff}
\end{align}

\noi
Consequently, we consider the following approximating dynamics
\begin{equation}\label{t-H-vk-flow}
	\begin{cases}
		\dt q_N  =  4\vk^2 \dx q_N - 4\vk^3 \pi_N \dx g_{-}(\vk, \pi_N q_N) , \\
		q_N (0) = q^0.
	\end{cases}
\end{equation}
Note that the solution $q_N$ to \eqref{t-H-vk-flow} decouples into $q_N = q_\text{low} + q_\text{high}$ which satisfy
\begin{align*}
	\dt q_\text{high} &= 4 \vk^2 \dx q_\text{high},\\
	\dt q_\text{low} & = 4 \vk^2 \dx q_\text{low} - 4\vk^3 \pi_N \dx g_{-}(\vk, q_\text{low}),
\end{align*}
with initial data $q_\text{high}(0)= \pi_{>N} q^0$ and $q_\text{low} (0) = \pi_N q^0$. From this, we see that the high frequencies of $q_N$ satisfy a transport equation, while the low frequencies evolve under the flow generated by the truncated Hamiltonian $H_{\vk,N}^\text{mKdV}$.

We start by establishing the global well-posedness of \eqref{t-H-vk-flow} and a good approximation property between the solutions to \eqref{t-H-vk-flow} and \eqref{Hkflow}.

\begin{lemma}\label{LM:trunc-flow-prop}
	Let $\vk\geq 1$ and $N\in\NB$. Then, the Cauchy problem \eqref{t-H-vk-flow} is globally well-posed in $L^2(\T)$. Denoting the data-to-solution map for \eqref{t-H-vk-flow} by $\Phi_{\vk, N}^{\operatorname{mKdV}}(t)$,  we have that
	\begin{align*}
		\| \Phi_{\vk, N}^{\operatorname{mKdV}}(t) \pi_N q^0 \|_{L^2} &= \| \pi_N q^0\|_{L^2} \qquad \text{and} \qquad
		\| \Phi_{\vk, N}^{\operatorname{mKdV}}(t) \pi_{>N} q^0\|_{L^2} = \|\pi_{>N} q^0\|_{L^2},
	\end{align*}
	for all $q^0\in L^2(\T)$.
	Moreover,
	for all $R_0>0$ there exists $T=T(R_0)>0$ such that
	\begin{equation}\label{approx-trunc}
		\lim_{N\to\infty}\sup_{\substack{0\leq t \leq T\\ \|q^0\|_{L^2} \leq R_0}} \| \Phi_{\vk}^\textup{mKdV}(t) q^0 - \Phi_{\vk, N}^\textup{mKdV}(t) q^0 \|_{L^2} =0.
	\end{equation}
\end{lemma}
\begin{proof}

	Let $\Phi_{\vk,N}^\text{mKdV}(t) q^0= q_\text{low}(t) + q_\text{high}(t)$. For the high frequencies, there exists a unique global solution given by $q_\text{high} (t) = S(t) \pi_{>N} q^0$ where the Fourier operator is defined by $\ft{S(t) f} (\xi) = e^{4it\vk^2 \xi} \ft{f}(\xi)$, for all $t\in\R$. Since $S(t)$ is an isometry in $L^2(\T)$, the $L^2$-conservation follows. For $q_\text{low}$, note that it satisfies the following Duhamel formulation
	\begin{align*}
		q_\text{low}(t) = S(t) \pi_N q^0 - 4\vk^3 \pi_N \int_{0}^{t} S(t-t') \dx g_{-}(\vk, q_\text{low}) (t') \, dt'
	\end{align*}
	and the local well-posedness in $L^2(\T)$ is a consequence of the boundedness of $S(t)$ in $L^2$ and the Lipschitz property of $g_{-}$ in \eqref{rLip}. Due to the smoothness of the data $\pi_N q^0$, the solutions $q_\text{low}$ are also smooth. Moreover, these solutions are global-in-time due to the conservation of the $L^2$-norm, which follows from the fact that $\pi_N q_\text{low} = q_\text{low}$, and \eqref{Grderiv}.

	To show \eqref{approx-trunc}, note that for $q^0 \in L^2(\T)$, we have that
	\begin{align*}
		&\| \Phi_{\vk}^{\operatorname{mKdV}}(t) q^0 - \Phi_{\vk, N}^{\operatorname{mKdV}}(t) q^0 \|_{L^2} \\
		& \leq \big\| \pi_N \big[ \Phi_{\vk}^{\operatorname{mKdV}}(t) q^0 - \Phi_{\vk, N}^{\operatorname{mKdV}}(t) q^0 \big] \big\|_{L^2} + \big\| \pi_{>N} \big[ \Phi_{\vk}^{\operatorname{mKdV}}(t) q^0 - \Phi_{\vk, N}^{\operatorname{mKdV}}(t) q^0 \big] \big\|_{L^2} \\
		& \leq C |t| \sup_{0\leq t' \leq t} \big\| \pi_N \dx \big[g_{-}(\vk,\Phi_\vk^\text{mKdV} (t')q^0) - g_{-}(\vk, \Phi_{\vk, N}^\text{mKdV}(t') \pi_Nq^0)  \big] \big\|_{L^2} \\
		& \quad + C |t| \sup_{0\leq t' \leq t} \big\| \pi_{>N} \dx g_{-}(\vk,\Phi_\vk^\text{mKdV} (t')q^0) \big\|_{L^2} \\
		& \leq C |t| \sup_{0\leq t' \leq t} \|\Phi_{\vk}^\text{mKdV}(t') q^0 - \Phi_{\vk, N}^\text{mKdV}(t') \pi_N q^0\|_{L^2} \|q^0\|_{L^2}^2 (1+\|q^0\|_{L^2})^6 \\
		& \quad + C |t| N^{-1} \sup_{0\leq t' \leq t} \|q^0\|_{L^2} ( 1+ \|q^0\|^4_{L^2})
	\end{align*}
	where we used \eqref{rLip}, \eqref{rbdL}, and the conservation of $L^2$-norm under the flows $\Phi_{\vk}^\text{mKdV}(t), \Phi^\text{mKdV}_{\vk, N} \pi_N(t)$, and $\Phi^\text{mKdV}_{\vk, N} \pi_{>N}(t)$.
	Also, we have
	\begin{align*}
		&\|\Phi_{\vk}^\text{mKdV}(t') q^0 - \Phi_{\vk, N}^\text{mKdV}(t') \pi_N q^0\|_{L^2}
		\leq \|\Phi_{\vk}^\text{mKdV}(t') q^0 - \Phi_{\vk, N}^\text{mKdV}(t')  q^0\|_{L^2} + \| \pi_{>N}q^0\|_{L^2}
	\end{align*}
	from the conservation of $L^2$-norm.
	Combining the above estimates and taking $T$ small enough such that $CT\|q^0\|_{L^2}^2(1+\|q^0\|_{L^2})^6\leq \frac12$, we get that
	\begin{align*}
		\sup_{0\leq t \leq T} \|\Phi_{\vk}^{\operatorname{mKdV}}(t) q^0 - \Phi_{\vk, N}^{\operatorname{mKdV}}(t) q^0 \|_{L^2} & \leq 2CT \|q^0\|_{L^2} (1+\|q^0\|_{L^2})^7 \big[ N^{-1} + \|\pi_{>N} q^0\|_{L^2}\big],
	\end{align*}
	from which \eqref{approx-trunc} follows.
\end{proof}

It can be proven that the quantity $\mathcal{E}_{s}(\kk,q)$ in  \eqref{Wkk} is conserved under the dynamics of any $H_{\vk}^{\text{mKdV}}$ flow for $\vk\geq 1$, at least for smooth solutions. Therefore, we expect that the measure $\rho_{s,R}$ is invariant under any $H_{\vk}^{\text{mKdV}}$ dynamics. To establish this, we first look at the truncated dynamics $H_{\vk,N}^\text{mKdV}$. However, the quantity $\mathcal{E}_{s}(\kk,\pi_N q_N)$, where $q_N$ denotes the solution of \eqref{t-H-vk-flow}, does not appear to be conserved; see \eqref{Aprojderiv}.
To bypass this difficulty, we proceed as in \cite{TV1,TV2} and instead prove asymptotic conservation for the truncated dynamics. This asymptotic result is sufficient for our goal of establishing the invariance of $\rho_{s,R}$ under the $H_{\vk}^\text{mKdV}$ dynamics.

\begin{proposition}\label{PROP:asym-cons}
	Let $\vk\geq 1$ and $q^0 \in L^2(\T)$. Then, there exists $0<\theta\ll1$ such that
	\begin{align*}
		\sup_{t\in\R} \bigg| \frac{d}{dt} \mathcal{E}_{s} \big(\Phi^\textup{mKdV}_{\vk, N}(t) \pi_N q^0 \big) \bigg|  \les N^{-\theta} \|q^0\|_{L^2}^6 ( 1+ \|q^0\|_{L^2}^8),
	\end{align*}
	for any $N\in \NB$.
\end{proposition}

\begin{proof}
	Let $u_N(t) = \Phi^\text{mKdV}_{\vk, N}(t)  \pi_N q^0 $, which is supported on frequencies $\{ \xi \in 2\pi\Z: \, |\xi| \leq 2\pi N\}$.
	Recall that
	\begin{align*}
		\mathcal{E}_{s} (u_N) = \int_1^\infty \kk^{2s} \A(\kk, u_N) \, d\kk= \int_1^\infty \kk^{2s} \Big[ A(\kk, u_N) - \tfrac12 A(\tfrac{\kk }{2}, u_N) \Big] \, d\kk,
	\end{align*}
	so we first focus on determining the time derivative of $A(\eta, u_N(t))$, for any fixed $\eta\geq \tfrac{1}{2}$.
	Using \eqref{derivA}, \eqref{Hkk-M-diff}, and \eqref{Grderiv}
	\begin{align*}
		\frac{d}{dt} A(\eta, u_N(t)) & = \Big\{ A(\eta), H_{\vk, N}^{\text{mKdV}} \Big\} \circ (u_N (t)) \\
		& = \int_{\T} g_{-}(\eta , u_N(t)) \dx \big[ 4 \vk^2 u_N(t) - 4\vk^3 \pi_N  g_{-}(\vk, u_N(t)) \big] \, dx \\
		& = 4\vk^2 \bigg[ 2 \eta \int_{\T} u_N(t) g_{+}(\eta, u_N(t)) \, dx - \vk \int_{\T} \pi_N g_{-}(\eta, u_N(t)) \dx \pi_N g_{-}(\vk, u_N(t)) \, dx \bigg] \\
		& = - 4 \vk^3 \int_{\T} \pi_N g_{-}(\eta, u_N(t)) \dx \pi_N g_{-}(\vk, u_N(t)) \, dx.
	\end{align*}
	Focusing on the inner integral, and omitting the dependence on $u_N(t)$ in the following, we have that
	\begin{align*}
		\int_{\T} \pi_N g_{-}(\eta) \dx \pi_N g_{-}(\vk) \, dx
		& = \frac12 \int_{\T}  \big[  \pi_N g_{-}(\eta) \dx \pi_N g_{-}(\vk)  - \pi_N g_{-}(\vk) \dx \pi_N g_{-}(\eta) \big] \, dx \\
		& =  \int_{\T}  \big[ \vk \pi_{>N} g_{-}(\eta)  \pi_{>N} g_{+}(\vk)  - \eta \pi_{>N} g_{-}(\vk)  \pi_{>N} g_{+}(\eta) \big] \, dx,
	\end{align*}
	where we used \eqref{Grderiv} and the fact that $\int [ \vk g_{-}(\eta) g_{+}(\vk) - \eta g_{-}(\vk) g_{+}(\eta) ] \,dx=0$. The latter follows from the differential identities in Proposition~\ref{PROP:Gamma}; for further
	details, see the proof of Proposition~3.1 in \cite{Forlano}.
	Using \eqref{prformula}, \eqref{Grderiv}, and the fact that $\pi_{>N} u_N =0$, we write
	\begin{align*}
		&\int_{\T}  \pi_N g_{-}(\eta) \dx \pi_N g_{-}(\vk) \, dx \\
		& = 8 \vk \eta \int_{\T}  \pi_{>N} \mathsf{R_0}(2\eta) \dx \big[ u_N \g(\eta) + u_N     \big] \pi_{>N} \mathsf{R_0}(2\vk)  \big[ u_N \g(\vk) + u_N  \big] \,dx \\
		& \qquad - 8 \eta \vk \int_{\T} \pi_{>N} \mathsf{R_0}(2\vk) \dx \big[ u_N \g(\vk) + u_N  \big] \pi_{>N} \mathsf{R_0}(2\eta) \big[ u_N \g(\eta)  + u_N \big] \,dx \\
		& = -16 \eta \vk \int_{\T}  \pi_{>N} \mathsf{R_0}(2\eta) \big[ u_N \g(\eta) \big] \pi_{>N} \dx \mathsf{R_0}(2\vk)\big[ u_N \g (\vk) \big] \, dx.
	\end{align*}
	Therefore,
	\begin{align}
		\frac{d}{dt} A(\eta, u_N) & =  4^3 \vk^4 \eta \int_{\T}  \pi_{>N} \mathsf{R_0}(2\eta) \big[ u_N \g(\eta) \big] \pi_{>N} \dx \mathsf{R_0}(2\vk)\big[ u_N \g (\vk) \big] \, dx. \label{Aprojderiv}
	\end{align}

	Now, we establish the decay of $\frac{d}{dt} A(\eta, u_N)$ in $N$ and $\eta$. Consider the decomposition $\frac{1}{4^3 \vk^4}\frac{d}{dt} A(\eta, u_N) = \mathsf{A}^{[2]}(\eta, u_N)  + \mathsf{A}^{[\geq 4]}(\eta, u_N) $, where
	\begin{align*}
		\mathsf{A}^{[2]}(\eta, u_N) &:= \eta \int_{\T}  \pi_{>N} \mathsf{R_0}(2\eta) \big[ u_N \g^{[2]}(\eta) \big] \pi_{>N} \dx \mathsf{R_0}(2\vk)\big[ u_N \g (\vk) \big] \, dx, \\
		\mathsf{A}^{[\geq 4]}(\eta, u_N) &:= \eta \int_{\T}  \pi_{>N} \mathsf{R_0}(2\eta) \big[ u_N \g^{[\geq 4]}(\eta) \big] \pi_{>N} \dx \mathsf{R_0}(2\vk)\big[ u_N \g (\vk) \big] \, dx.
	\end{align*}
	For the first contribution, by Cauchy-Schwarz, \eqref{LinftyH1k}, Bernstein's inequality, and \eqref{gbd}, we have that
	\begin{align*}
		\big| \mathsf{A}^{[2]}(\eta, u_N) \big| & \leq \eta \big\| \pi_{>N} \mathsf{R_0}(2\eta) \big[ u_N \g^{[2]}(\eta) \big] \big\|_{H^\dl} \big\| \pi_{>N} \dx \mathsf{R_0}(2\vk)\big[ u_N \g (\vk) \big] \big\|_{H^{-\dl}} \\
		& \les \eta^{-1+\dl} N^{-1-\dl}\| u_N \g^{[2]} (\eta) \|_{L^2}  \| u_N \g(\vk) \|_{L^2} \\
		& \les \eta^{-1+\dl} N^{-1-\dl}\| u_N \|_{L^\infty} \|\g^{[2]} (\eta) \|_{L^2} \| u_N \|_{L^2} \| \g(\vk) \|_{L^\infty} \\
		& \les \eta^{-2+\dl} N^{-1-\dl}  \|u_N\|_{L^\infty} \|u_N\|_{L^2}^2 \big\| \tfrac{u_N}{2\vk \pm \partial}\big\|_{L^\infty} \|\g(\vk)\|_{H^1_\vk}  \\
		& \les \eta^{-3+\dl} N^{\eps-\dl} \| \pi_N q^0\|^6_{L^2} (1+ \|\pi_N q^0\|_{L^2}^4),
	\end{align*}
	for any $\eps>0$ and any $0<\dl<1$,
	where the last line follows from conservation of the $L^2$-norm. For the second contribution, we have that
	\begin{align*}
		\big| \mathsf{A}^{[\geq 4]}(\eta, u_N) \big| & \leq \eta \|\pi_{>N} \mathsf{R_0}(2\eta) \big[ u_N \g^{[\geq 4]}(\eta) \big] \|_{L^2} \|\pi_{>N} \dx \mathsf{R_0}(2\vk)\big[ u_N \g (\vk) \big] \|_{L^2} \\
		& \les \eta^{-1}N^{-1} \| u_N \g^{[\geq 4]} (\eta) \|_{L^2}  \| u_N \g(\vk) \|_{L^2} \\
		& \les \eta^{-\frac{7}{2}} N^{-\frac14} \|\pi_N q^0\|_{L^2}^{8} ( 1+ \|\pi_N q^0\|_{L^2}^6)
	\end{align*}
	using Bernstein's inequality, \eqref{gbd}, and \eqref{gbd2}.
	Combining the two inequalities, we get
	\begin{align*}
		& \bigg| \int_1^\infty \kk^{2s} \frac{d}{dt} \A(\kk , u_N(t)) \, d\kk \bigg| \\
		& \les \int_1^\infty \kk^{2s} \bigg[  \kk^{-3+\dl} N^{\eps-\dl} \| q^0\|^6_{L^2}(1+\| q^0 \|^4_{L^2}) + \kk^{-\frac72} N^{-\frac14} \|q^0\|_{L^2}^8 (1+\|q^0\|_{L^2}^6)\bigg] \, d\kk \\
		& \les N^{\eps-\dl} \|q^0\|^6_{L^2} ( 1 + \|q^0\|^8_{L^2}) \int_1^\infty \kk^{2s-3+\dl} \, d\kk \\
		& \les N^{-\frac\dl2} \|q^0\|^6_{L^2} ( 1 + \|q^0\|^8_{L^2})
	\end{align*}
	given that $2s-3+\dl<-1 \iff \dl<2(1-s)$ and by picking $\eps= \frac\dl2$.
	Since the estimate above is independent of $t$, we obtain that
	\begin{align*}
		\bigg| \frac{d}{dt}\int_1^\infty \kk^{2s} \A(\kk , u_N(t))  \, d\kk\bigg| & = \bigg| \int_1^\infty \kk^{2s} \frac{d}{dt} \A(\kk, u_N(t))  \, d\kk\bigg|
		\les  N^{-\frac\dl2}  \|q^0\|^6_{L^2} (1+ \|q^0\|_{L^2}^8),
	\end{align*}
	as intended.
\end{proof}

\subsection{Proof of invariance under $H_{\vk}^{\text{mKdV}}$ and $H^{\text{mKdV}}$}

Before proceeding to the proof of invariance of $\rho_{s,R}$ under the $H_{\vk}^{\text{mKdV}}$ dynamics, we establish some preliminary results. Let $\mathcal{B}(L^2)$ denote the set of Borel sets in $L^2(\T)$ and $B(R)$ denote the ball of radius $R^2>0$ in $L^2(\T)$. Also, let $E_N = \operatorname{span}(\cos(\xi x), \sin(\xi x))_{|\xi| \leq 2\pi N}$ and $E_N^\perp = \operatorname{span}(\cos(\xi x), \sin(\xi x))_{|\xi| > 2\pi N}$.

\begin{lemma}\label{LM:inv-finite-base}
	The flow $\Phi_{\vk,N}^{\operatorname{mKdV}}(t)\pi_{>N}$ leaves $E^\perp_N$ and the measure $\mu_{s,N}^\perp$ invariant, while the flow $\Phi_{\vk,N}^{\operatorname{mKdV}}(t)\pi_N$ leaves $E_N$ and the Lebesgue measure $dq_0 dq_1 \cdots dq_N$ invariant.
\end{lemma}
\begin{proof}
	Note that $\mu^\perp_{s,N}$ is the probability measure on $E^\perp_N$ induced by the map
	\begin{align*}
		\o\in\O \mapsto  \sum_{|\xi|>2\pi N} \frac{g_\xi(\o)}{\jb{\mathfrak{m}_s(\xi)}} e^{i\xi x}.
	\end{align*}
	Moreover, the random variable obtained by applying $S(t)$ to this Fourier series has the same law as the original random variable. Then, $S(t) =\pi_{>N}\Phi^\text{mKdV}_{\vk, N}(t)$ leaves $\mu^\perp_{s,N}$~invariant.

	For the invariance under $\Phi_{\vk, N}^\text{mKdV}(t)\pi_N$, we consider the finite-dimensional system of ODEs for the Fourier coefficients of the solution and establish that this is a divergence free system. The result then follows by Liouville's theorem. Let $u_N(t) = \Phi^\text{mKdV}_{\vk, N}(t) \pi_N q^0$ and write
	$$u_N(t,x)= a_0(t) + \sum_{1\leq m \leq N} a_m(t) (e_m(x) + e_{-m}(x)) + \sum_{1\leq m \leq N} i b_m(t) (e_{m}(x) - e_{-m}(x)),$$
	where $e_m(x) = e^{2\pi i m x}$ for $m\in\Z$.
	Then, the Fourier coefficients of $q_N$ satisfy $\dt a_0 =0$ and
	\begin{align*}
		\dt a_m &= -8\pi\eta^2 mb_m + 8\pi \vk^3 m \Im \int_\T g_{-}(x;\vk, u_N[a,b]) e_{-m}(x) \, dx =: F_{a_m}(a,b),\\
		\dt b_m &= 8\pi\eta^2 ma_m - 8\pi\vk^3 m \Re \int_\T g_{-}(x;\vk, u_N[a,b]) e_{-m}(x) \, dx =: F_{b_m}(a,b),
	\end{align*}
	where $(a,b) = (a_0,\ldots, a_N, b_1, \ldots, b_N)$.
	In the following, we will drop the dependence of $g_-$ on $\vk$. It suffices to show
	\begin{align*}
		\frac{\partial F_{a_m} }{\partial a_m} + \frac{\partial F_{b_m}}{\partial b_m} =0 , \quad \text{for all} \quad 1 \leq m \leq N.
	\end{align*}
	Fix $1 \leq m \leq N$ and note that
	\begin{align*}
		\frac{\partial F_{a_m}}{\partial a_m} &= 8\pi\vk^3 m \Im \int_\T \frac{\partial }{\partial a_m} g_{-}(x; u_N[a,b]) e_{-m}(x) \, dx, \\
		\frac{\partial F_{b_m}}{\partial b_m} &= -8\pi\vk^3 m \Re \int_\T \frac{\partial }{\partial b_m} g_{-}(x; u_N[a,b]) e_{-m}(x)\, dx.
	\end{align*}
	Using the fact that
	\begin{equation*}
		\begin{aligned}
			\frac{\partial }{\partial a_m} g_{-} (x;q[a,b]) &= \frac{d}{d\theta} g_{-}\big[x; u_N[(a,b) + \theta v_m]  \big] \Big\vert_{\theta=0}  = \int_\T \frac{\dl g_{-}(x)}{\dl u_N[a,b](y)} \Big< \frac{\dl u_N[a,b](y)}{\dl v_m}, v_m \Big> dy , \\
			\frac{\partial }{\partial b_m} g_{-} (x;u_N[a,b]) &= \frac{d}{d\theta} g_{-}\big[x;u_N[(a,b) + \theta v'_m]\big] \Big\vert_{\theta=0}  = \int_\T \frac{\dl g_{-}(x)}{\dl u_N[a,b](y)} \Big< \frac{\dl u_N[a,b](y)}{\dl v'_m}, v'_m \Big> dy,
		\end{aligned}
	\end{equation*}
	where the vectors $v_m, v'_m \in \R^{2N+1}$ have entry equal to 1 in the position corresponding to $a_m$ and $b_m$, respectively, and are 0 otherwise, and $\jb{\cdot, \cdot}$ denotes the inner product in $\R^{2N+1}$.
	For the directional derivative of $u_N$, we have that
	\begin{align}
		\Big< \frac{\dl u_N[a,b](y)}{\dl v_m}, v_m \Big> & = e_m(y) + e_{-m}(y), \quad
		\Big< \frac{\dl u_N[a,b](y)}{\dl v'_m}, v'_m \Big>  = i(e_m(y) - e_{-m}(y)). \label{q-directional-deriv}
	\end{align}

	\noi From \eqref{rderivq} and \eqref{detG}, we have
	\begin{align}
		\frac{\dl g_{-}(x)}{\dl u_N[a,b](y)} & =G_{11}^{2}(x,y)+G_{22}^{2}(x,y)-G_{12}^{2}(x,y)-G_{21}^{2}(x,y)
		=:\mathsf{G}(x,y), \label{r-func-aux}
	\end{align}
	Using \eqref{q-directional-deriv} and \eqref{r-func-aux}, gives
	\begin{align*}
		\frac{\partial }{\partial a_m} g_{-}(x; u_N[a,b]) & = - \int_\T  \mathsf{G}(x,y) [e_m(y) + e_{-m}(y)] \, dy , \\
		\frac{\partial }{\partial b_m} g_{-}(x; u_N[a,b]) & = - \int_\T  \mathsf{G}(x,y) i [e_m(y) - e_{-m}(y)] \, dy,
	\end{align*}
	from which we get
	\begin{align*}
		&\frac{1}{8\pi \vk^3 m}\Big(	\frac{\partial F_{a_m}}{\partial a_m} + \frac{\partial F_{b_m}}{\partial b_m} \Big) \\
		& = \Im \bigg(- \iint_{\T^2 }  \mathsf{G}(x,y)[e_m(y) + e_{-m}(y)]  e_{-m}(x) \, dy \, dx  \bigg) \\
		& \quad - \Re \bigg( - \iint_{\T^2 } \mathsf{G}(x,y) i [e_{m}(y) - e_{-m}(y)]  e_{-m}(x) \, dy \, dx \bigg) \\
		& = - \iint_{\T^2 } \mathsf{G}(x,y) \Im \Big[ [e_{m}(y) + e_{-m}(y)]  e_{-m}(x) + [e_{m}(y) - e_{-m}(y)] e_{-m}(x) \Big] \, dx \, dy \\
		& = i  \iint_{\T^2 }  \mathsf{G}(x,y) [e_m(x-y) - e_{m}(y-x)] \, dx \, dy.
	\end{align*}

	\noi By the symmetries of $G$ in \eqref{Gsyms}, we see that  $\mathsf{G}(y,x)=\mathsf{G}(x,y)$ and hence
	\begin{align*}
		\frac{1}{8\pi\vk^3 m}\Big(	\frac{\partial F_{a_m}}{\partial a_m} + \frac{\partial F_{b_m}}{\partial b_m} \Big) = i \iint_{\T^2 }  [\mathsf{G}(x,y) - \mathsf{G}(y,x) ]  e_m(x-y)\, dx \, dy=0
	\end{align*}
	as intended.
\end{proof}

We can now prove the following change of variables formula.
\begin{lemma}\label{LM:measure-change-var}
	For $\vk \geq 1$, $R>0$, $N\in\NB$, $t\in\R$, and $A\in\mathcal{B}(L^2)$ , we have that
	\begin{align*}
		\rho_{s,R,N}( \Phi_{\vk,N}^{\operatorname{mKdV}}(t)(A) ) = \int_A \ind_{ \{\|q\|^2_{L^2} \leq R\}  } \exp\big( - \mathcal{E}_{s} (\Phi^{\operatorname{mKdV}}_{\vk,N}(t) \pi_N q) \big) \, dq_0 dq_1 \cdots d q_N \otimes d\mu_{s,N}^\perp.
	\end{align*}
\end{lemma}
\begin{proof}
	Let $ dL_N$ denote the Lebesgue measure $dq_0 \cdots dq_N$.
	Using Fubini's theorem, Lemma~\ref{LM:inv-finite-base}, and the invariance of $\mu_{s,N}^{\perp}$ under $S(t)$, we have that
	\begin{align*}
		&\rho_{s,R,N}( \Phi_{\vk,N}^{\operatorname{mKdV}}(t)(A) ) \\
		& =  \int_{\Phi_{\vk,N}^{\operatorname{mKdV}}(t)(A)} \ind_{\{\|q\|^2_{L^2} \leq R\}}\  e^{-\mathcal{E}_{s}(\pi_N q)} \, dL_N \otimes d\mu_{s, N}^\perp \\
		& = \int_{L^2} \ind_{\Phi_{\vk,N}^{\operatorname{mKdV}}(t)(A)} (q) \ind_{\{\|q\|^2_{L^2} \leq R\}} e^{- \mathcal{E}_{s} (\pi_N q)} \, dL_N \otimes d \mu_{s, N}^\perp \\
		& = \int_{E_N } \bigg[ \int_{E_N^\perp} \ind_{\Phi_{\vk,N}^{\operatorname{mKdV}}(t)(A)} (\pi_N q , S(t) \pi_{>N} q) \ind_{\{\|\pi_N q + S(t) \pi_{>N} q\|^2_{L^2} \leq R\}} e^{- \mathcal{E}_{s}(\pi_N q)} d\mu_{s, N}^\perp \bigg] dL_N \\
		& = \int_{E^\perp_N} \bigg[ \int_{E_N} \ind_{\Phi_{\vk,N}^{\operatorname{mKdV}}(t)(A)} (\Phi_{\vk,N}^{\operatorname{mKdV}}(t)q) \ind_{\{\| \Phi_{\vk,N}^{\operatorname{mKdV}}(t)q \|^2_{L^2} \leq R\}} e^{-\mathcal{E}_{s}(\Phi_{\vk,N}^{\operatorname{mKdV}}(t)\pi_N q)} dL_N\bigg] d\mu_{s, N}^\perp.
	\end{align*}
	The result now follows from the conservation of $L^2$-norm under $\Phi_{\vk,N}^{\operatorname{mKdV}}(t)$ and the fact that $\Phi_{\vk,N}^{\operatorname{mKdV}}(t)$ is a bijection on $L^2(\T)$.
\end{proof}

\begin{lemma}\label{LM:M-inv-limit}
	Let $\vk\geq 1$. Then,
	\begin{align}\label{dt-decays}
		\lim_{N\to\infty} \sup_{\substack{t\in \R \\ A \in \mathcal{B}(L^2)}} \bigg| \frac{d}{dt} \int_{\Phi^{\operatorname{mKdV}}_{\vk,N}(t) (A)} F_N(q) \, d\mu_{s} \bigg| = 0.
	\end{align}
	Moreover, for all $\vk \geq 1$,  $A\in\mathcal{B}(L^2)$, and $t\in\R$, we have that
	\begin{align}\label{t-diff-decays}
		\lim_{N\to\infty} \bigg| \int_{\Phi^{\operatorname{mKdV}}_{\vk,N}(t)(A)} F_N(q) \, d\mu_{s} - \int_A F_N(q) \, d\mu_{s} \bigg| =0.
	\end{align}
\end{lemma}
\begin{proof}
	Fix $t\in\R$, $A\in \mathcal{B}(L^2)$, and let $dL_N = dq_0 \cdots dq_N$. From Lemma~\ref{LM:measure-change-var} and Proposition~\ref{PROP:asym-cons}, we have that
	\begin{align*}
		&\bigg|\frac{d}{dt}	\int_{\Phi_{\vk, N}^\text{mKdV}(t)(A)} F_N(q) \, d\mu_{s} \bigg| \\
		&= \bigg|\frac{d}{dt}  \int_A \ind_{ \{\|q\|^2_{L^2} \leq R\}  } e^{-  \mathcal{E}_{s} (\Phi^{\operatorname{mKdV}}_{\vk,N}(t) \pi_N q)} \, dL_N \otimes d\mu_{s,N}^\perp \bigg| \\
		& \leq  \int_A \bigg| \frac{d}{dt}  \mathcal{E}_{s} (\Phi^{\operatorname{mKdV}}_{\vk,N}(t) \pi_N q) \bigg| \ind_{ \{\|q\|^2_{L^2} \leq R\}  } e^{-  \mathcal{E}_{s} (\Phi^{\operatorname{mKdV}}_{\vk,N}(t) \pi_N q)} \, dL_N \otimes d\mu_{s,N}^\perp \\
		& \les \vk^4 R^3(1+R^4) N^{-\theta}
	\end{align*}
	for some small $\theta=\theta(s)>0$,
	where the implicit constants are independent of $t$ and $A$. From the above, we obtain the decay in $N$,  uniformly in $t$ and $A$, and the limit in \eqref{dt-decays} follows. The second result \eqref{t-diff-decays} follows from \eqref{dt-decays} and the Fundamental Theorem of Calculus.
\end{proof}

\begin{lemma}\label{LM:invariant-vk-compact}
	Fix $\vk\geq 1$ and $t\in\R$. Then,  for all compact sets $K\subset L^2(\T)$, we have that
	\begin{align*}
		\int_K F(q) \, d\mu_{s} = \int_{\Phi^{\operatorname{mKdV}}_{\vk}(t)(K)} F(q) \, d\mu_{s}.
	\end{align*}
\end{lemma}
\begin{proof}
	Fix $\vk\geq 1$, $t\in\R$, and a compact set $K\subset L^2(\T)$. Let $R>0$ be the parameter associated with the $L^2$-cut-off in $F$.

	We first show that there exists $T=T(R)>0$ such that for $t\in[-T,T]$ we have
	\begin{align}\label{compact-ineq1}
		\int_K F(q) d\mu_{s} \leq \int_{\Phi_\vk^\text{mKdV}(t)(K)} F(q) \, d\mu_{s}.
	\end{align}
	From Lemma~\ref{LM:M-inv-limit} and the $L^1(L^2(\T);d\mu_{s})$ convergence of the densities $F_N$ in \eqref{LpconvFM}, we have
	\begin{align}
		\lim_{N\to\infty} \int_{\Phi^\text{mKdV}_{\vk,N}(t)(K)} F_N(q) \, d\mu_{s}
		& = \int_K F(q) \, d\mu_{s}. \label{lim-M-ineq-aux1}
	\end{align}

	We now prove a local-in-time upper bound for the left-hand side of \eqref{lim-M-ineq-aux1}.
	From \eqref{approx-trunc},
	for fixed $\eps>0$, there exists $N_0 \in \NB$ such that for all $t\in[-T,T]$ and all $N\geq N_0$,
	\begin{equation}\label{approx-sets}
		\Phi_{\vk,N}^\text{mKdV}(t) (K \cap B(R)) \subset \Phi_\vk^\text{mKdV}(t) (K\cap B(R)) + B(\eps).
	\end{equation}
	Note that $\Phi^\text{mKdV}_{\vk, N}(t)(K) \cap B(R) \subseteq \Phi^\text{mKdV}_{\vk, N}(t)(K \cap B(R))$.
	Therefore, by \eqref{approx-sets}, we get that for all $t\in[-T,T]$ and $N\geq N_0$
	\begin{align}
		\int_{\Phi^\text{mKdV}_{\vk, N}(t)(K)} F_N(q) \, d\mu_{s} & = \int_{\Phi^\text{mKdV}_{\vk, N}(t)(K) \cap B(R)} F_N(u) \, d\mu_{s} \nonumber\\
		& \leq \int_{\Phi^\text{mKdV}_{\vk, N}(t)(K \cap B(R) )} F_N(q) \, d\mu_{s}\nonumber \\
		& \leq  \int_{\Phi_{\vk}^\text{mKdV}(t)(K\cap B(R)) + B(\eps)} F_N(q) \, d\mu_{s}. \label{lim-M-ineq-aux2}
	\end{align}
	Since $\Phi_{\vk}^\text{mKdV}(t)$ is a diffeomorphism on $L^2(\T)$ and $K\cap B(R)$ is closed in $L^2(\T)$, the image of $K\cap B(R)$ is also closed and
	\begin{align}
		\bigcap_{\eps>0} \big( \Phi_{\vk}^\text{mKdV}(t)(K\cap B(R)) + B(\eps) \big) = \Phi_\vk^\text{mKdV}(t)(K \cap B(R)). \label{epsintersect}
	\end{align}
	Therefore, by continuity from above of the probability measure $\mu_{s}$, \eqref{epsintersect} and the fact that $F(q)\in~L^1(L^2(\T);d\mu_{s})$, it follows that
	\begin{align}\label{lim-eps-F}
		\lim_{\eps\to 0} \int_{\Phi_{\vk}^\text{mKdV}(t)(K\cap B(R)) + B(\eps)} F(q) \, d\mu_{s} = \int_{\Phi_{\vk}^\text{mKdV}(t)(K\cap B(R))} F(q) \, d\mu_{s}.
	\end{align}
	Consequently, combining \eqref{lim-M-ineq-aux1}, \eqref{lim-M-ineq-aux2}, \eqref{lim-eps-F}, and \eqref{LpconvFM}, by first taking a limit as $N\to \infty$ followed by a limit as $\eps\to0$, we obtain that for $t\in[-T,T]$
	\begin{align*}
		\int_K F(u) \, d\mu_{s} & =
		\lim_{N\to\infty} \int_{\Phi^\text{mKdV}_{\vk,N}(t)(K)} F_N(q) \, d\mu_{s} \\
		& \leq \lim_{\eps\to0}\lim_{N\to\infty} \int_{\Phi_{\vk}^\text{mKdV}(t)(K\cap B(R)) + B(\eps)} [F_N(q) - F(q)]  \, d\mu_{s} \\
		& \qquad + \lim_{\eps \to 0}	\lim_{N\to\infty}  \int_{\Phi_{\vk}^\text{mKdV}(t)(K\cap B(R)) + B(\eps)} F(q)  \, d\mu_{s} \\
		& = \int_{\Phi_{\vk}^{\text{mKdV}}(t) (K\cap B(R))} F(q) \, d\mu_{s} \\
		& \leq \int_{\Phi_{\vk}^{\text{mKdV}}(t) (K)} F(q) \, d\mu_{s}.
	\end{align*}

We now extend \eqref{compact-ineq1} to all times $t\in\R$. Since $\Phi_{\vk}^\text{mKdV}(T)$ is a diffeormophism in $L^2(\T)$ and  \eqref{compact-ineq1} applies to any compact set in $L^2(\T$), for all $t\in[-T,T]$
	\begin{align*}
		\int_{\Phi^\text{mKdV}_{\vk}(T) (K) } F(q) \, d\mu_{s}  & \leq  \int_{\Phi^\text{mKdV}_{\vk}(t) ( \Phi_{\vk}^\text{mKdV}(T)( K)) } F(q) \, d\mu_{s}  =
		\int_{\Phi^\text{mKdV}_{\vk}(t+T) (K) } F(q) \, d\mu_{s},
	\end{align*}
	which combined with \eqref{compact-ineq1} extends the inequality to $t\in[0,2T]$. By repeating this process, we show that the inequality holds for all $t\geq 0$. The same argument applies for $t<0$.

Lastly, let $t\in\R$ and $K$ be a compact set in $L^2(\T)$. Then, $\wt{K}:=\Phi^\text{mKdV}_{\vk}(t)(K)$ is a compact set in $L^2(\T)$. Using \eqref{compact-ineq1} and the group property of $\Phi^\text{mKdV}_\vk(t)$, we obtain
	\begin{align*}
		\int_{\Phi^\text{mKdV}_{\vk}(t)(K)} F(q) \, d \mu_{s} = \int_{\wt{K}} F(q) \, d\mu_{s} \le \int_{\Phi^\text{mKdV}_{\vk}(-t)(\wt{K})} F(q) \, d \mu_{s} = \int_K F(q) \, d\mu_{s}.
	\end{align*}
	The equality follows by combining the above with \eqref{compact-ineq1}.
\end{proof}

We now prove the invariance of $\mu_{s}$ under the $H^\text{mKdV}_\vk$-flow.
\begin{proposition}\label{PROP:Hkinv}
	Fix $\vk\geq 1$ and $t\in\R$. Then, for all $A\in\mathcal{B}(L^2)$ we have that
	\begin{align*}
		\int_A F(q) \, d\mu_{s} = \int_{\Phi^\textup{mKdV}_\vk(t)(A)} F(q) \, d\mu_{s},
	\end{align*}
	that is, the measure $\rho_{s,R}$ is invariant under the $H^\textup{mKdV}_{\vk}$-flow.
\end{proposition}
\begin{proof}
	Let $A\in \mathcal{B}(L^2)$. Then, there exists a sequence of compact sets $K_n \subset A$ such that
	\begin{align*}
		\lim_{n\to\infty} \int_{K_n} F(q) \, d\mu_{s} = \int_A F(q) \, d\mu_{s}.
	\end{align*}
	From Lemma~\ref{LM:invariant-vk-compact}, we have that for all $t\in\R$,
	\begin{align*}
		\int_{K_n} F(q) \, d\mu_{s} &= \int_{\Phi_{\vk}^\text{mKdV}(t)(K_n)} F(q) \, d\mu_{s} \leq \int_{\Phi_\vk^\text{mKdV}(t)(A)} F(q) \, d\mu_{s}
	\end{align*}
	from the inclusion $\Phi_\vk^\text{mKdV}(t)(K_n) \subset \Phi_\vk^\text{mKdV}(t)(A)$. Therefore,
	\begin{align*}
		\int_A F(q) \, d\mu_{s} \leq \int_{\Phi_\vk^\text{mKdV}(t)(A)} F(q) \, d\mu_{s}.
	\end{align*}
	The reverse inequality follows from the above estimate with $A\mapsto \Phi_\vk^\text{mKdV}(t)(A)$, $t\mapsto -t$, and from the group property of the flow.
\end{proof}

\begin{proof}[Proof of Theorem~\ref{THM:mkdv}(iii)]
	It suffices to prove \eqref{mkdvinv} for a class of test functions which is dense in $L^1(L^2(\T); d\mu_{s})$. We thus may assume that $f$ is continuous and bounded. By the approximation property \eqref{mkdvapprox}, dominated convergence, and Proposition~\ref{PROP:Hkinv}, we have
	\begin{align*}
		\int_{L^2(\T)}f( \Phi^{\text{mKdV}}(t)q) \, d\rho_{s,R}& = \lim_{\vk \to \infty} \int_{L^2(\T)}f( \Phi^{\text{mKdV}}_{\vk}(t)q)\, d\rho_{s,R} \\
		 &= \lim_{\vk \to \infty} \int_{L^2(\T)}f( q)\,d\rho_{s,R}
		  = \int_{L^2(\T)}f( q) \, d\rho_{s,R}.
	\end{align*}
	This completes the proof of \eqref{mkdvinv}.
\end{proof}

We complete this section by proving Theorem~\ref{THM:mkdv}(ii) on the union in $R$ of the support of the weighted measures $\rho_{s,  R}$. We follow the argument in \cite{TV1}.
\begin{proof}[Proof of Theorem~\ref{THM:mkdv}(ii)]
	Note that since $\|q\|_{L^2}<\infty$ a.s. in the support of the measure~$\mu_{s}$, we have that
	$\lim\limits_{R\to\infty} \ind_{\{\|q\|^2_{L^2} \leq R\}} = 1$
	$\mu_{s}$-a.e.
	Thus, by Egorov's theorem, for all $n\in \NB$, $n\geq 2$,
	there exist measurable sets $\O_n\subset L^2(\T)$, and $\wt{R}_n>0$ such that $\mu_{s}(\O_n) = 1-\frac1n$ and $\ind_{\{\|q\|^2_{L^2} \leq R\}}=1$ $\mu_{s}$-a.e. for $R>\wt{R}_n$ and $q\in \O_{n}$. Therefore,
	\begin{align*}
		\bigcup_{R>0} \supp (\ind_{\{\|q\|^2_{L^2} \leq R\}}) \supseteq \bigcup_{n\geq 2} \O_n
	\end{align*}
	and $\lim_{n\to\infty} \mu_{s}(\O_n) = 1$, from which we conclude that $$\mu_{s}\bigg(\bigcup_{R>0} \supp (\ind_{\{\|q\|^2_{L^2} \leq R\}})\bigg)=1.$$

	Lastly, Lemma~\ref{LEM:EJM} guarantees that $E$ is in $L^p(d\mu_{s})$, thus the exponential component of the density $F(q)$ in $\rho_{s,R}$ is $\mu_{s}$-a.s. non-zero.
	Combining these two results, we have that
	\begin{align*}
		\mu_{s} \bigg( \bigcup_{R>0} \supp\Big(\ind_{\{\|q\|_{L^2}^2 \leq R\}} e^{-E(q)}
		\Big)\bigg)  = 1
	\end{align*}
	and thus the union of the supports of $\rho_{s,R}$ with respect to $R>0$ agrees with the support of the base Gaussian measure $\mu_{s}$.
\end{proof}

\subsection{Invariant measures for KdV}\label{SEC:kdv}

In this subsection, we prove Theorem~\ref{THM:kdv}.

\begin{proof}[Proof of Theorem~\ref{THM:kdv}]
Fix $\tfrac{1}{2}<s<1$ and $R>0$.
We first note that by conservation of the $L^2$-norm and the mean, we have
\begin{align}
 \Phi^{\text{mKdV}}(t)(L^2_{0}(\T))=L^2_{0}(\T) \quad \text{for all} \,\, t\in \R, \label{mkdvinvariances2}
\end{align}
and $\int \Phi^{\text{KdV}}_{0}(t)(w^0) \, dx =0$
for all $t\in \R$ and $w^0\in H^{-1}_{0}(\T)$. Thus, $\Phi_{0}^{\text{KdV}}$ is well-defined as a map from $L^2_{0}(\T)$ to $H^{-1}_{0}(\T)$.

We first prove (i). Fix $f\in L^1(H^{-1}_{0}(\T);d\nu^0)$. Then, by the definition of the push-forward measure, \eqref{mkdvinvariances2}, and the invariance of $\rho^0$ under $\Phi^{\text{mKdV}}(t)$ (Theorem~\ref{THM:mkdv}), we have
\begin{align*}
 \int_{H^{-1}_{0}(\T)} f( w^0)d\nu^{0}(w^0)
 & = \int_{L^2_{0}(\T)} f(B(q^0))d\rho^{0}(q^0) \\
 & =  \int_{L^2(\T)} \ind_{L^2_{0}(\T)}(\Phi^{\text{mKdV}}(t)(q^0)) f(B \circ \Phi^{\text{mKdV}}(t)(q^0))d\rho^{0}(q^0) \\
 & =\int_{L^2_0(\T)}  f(B \circ \Phi^{\text{mKdV}}(t)(q^0))d\rho^{0}(q^0) \\
 & =\int_{H^{-1}_{0}(\T)} f(\Phi^{\text{KdV}}_{0}(w^0))d\nu^0(w^0).
\end{align*}

For the proof of (ii), we argue similarly noting that $\tau_{-\al}(H^{-1}_{\al}(\T))=H^{-1}_{0}(\T)$.

Lastly, we show (iii). It suffices to take $\al=0$. Note that $\Phi^{\textup{KdV}}_{0}=\Phi^{\textup{KdV}}_{S} \vert_{H^{-1}_{0}(\T)}$, where $\Phi_{S}^{\text{KdV}}$ is the solution map constructed in \cite{KapTopKdv, KV}. Let $w^0\in H^{-1}_{0}(\T)$, $q^0:=B^{-1}(w^0)\in L^2_{0}(\T)$, and $(q_{n}^{0})$ be a sequence of smooth, mean-zero functions converging to $q^0$ in $L^2(\T)$, such as $q_{n}^{0}=\pi_{n}q^0$.
Define the sequence of smooth functions $w_{n}^{0}:=B(q_{n}^{0})$. By continuity of $B$, $w_{n}^{0}$ converges to $w^0$ in $H^{-1}(\T)$. Let $w_{n}(t)$ denote the corresponding smooth solutions to KdV.
On the one hand, the results in \cite{KapTopKdv, KV}, imply $w_{n}(t)\to \Phi^{\text{KdV}}_{S}(t)(w^0)$ in $H^{-1}(\T)$ for each $t\in \R$. On the other hand, by the continuity of $B$, $\Phi^{\text{KdV}}(t)(w^{0}_{n})\to \Phi^{\text{KdV}}(t)(w^{0})$ in $H^{-1}(\T)$. Now, since each $w_{n}^0$ are smooth, $w_{n}(t)=\Phi^{\text{KdV}}(t)(w_{n}^{0})$ and thus we conclude that $\Phi^{\text{KdV}}(t)(w^0)~=~\Phi_{S}^{\text{KdV}}(w^0)$.
\end{proof}

\section{On the defocusing cubic NLS} \label{SEC:NLS}

In this section, we detail the key modifications required to prove Theorem~\ref{THM:NLS} below, which is an analogue of Theorem~\ref{THM:mkdv} for the defocusing cubic NLS equation \eqref{NLS}. As many of the estimates follow along similar lines to what we have shown previously, we will omit the corresponding proofs here.

For the defocusing cubic NLS, the Lax operator is
\begin{align*}
\mathcal{L}_{\text{NLS}} =
\begin{bmatrix}
-\dd & q \\
-\cj{q} & \dd
\end{bmatrix}.
\end{align*}
For $q\in L^2(\T)$, the Kato-Rellich theorem implies that $\mathcal{L}$ is an unbounded anti-self-adjoint operator with domain $H^{1}(\R)\times H^{1}(\R)$. Thus, the resolvent $R_{\text{NLS}}(\kk):= (\mathcal{L}_{\text{NLS}}+\kk)^{-1}$ exists for all $|\kk|\geq \tfrac{1}{2}$ and it is jointly analytic on $\{ (\kk, q):\, |\kk|\geq \tfrac{1}{2}, \,\, q\in L^2(\T)\}$. The resolvent also admits an integral kernel, which we still denote by $G(x,y;\kk)$, and satisfies \eqref{Gcts} replacing $R(\kk)$ with $R_{\text{NLS}}(\kk)$. The Green's function $G$ is a bounded continuous function on $\{ (x,y)\in \R^2 \,:\, x\neq y\}$, with all components decaying to zero as $|x|\to \infty$ for fixed $y\in \R$. In particular, \eqref{detG} holds.
In view of the symmetry property $\mathcal{L}_{\text{NLS}}(\kk)^{\ast} = -\mathcal{L}_{\text{NLS}}(-\kk)$, the Green's function satisfies
\begin{align}
G(x,y;-\kk) = -G(y,x; \kk). \label{symGNLS}
\end{align}
Following \cite{HGKV}, we define
\begin{align*}
\g(x;\kk,q) : =  \sgn(\kk) [ (G-G_0)_{11}&(x,x;\kk)+ (G-G_0)_{22}(x,x;\kk)],\\
g_{12}(x;\kk,q) : = \sgn(\kk) G_{12}(x,x;\kk) \quad &\text{and} \quad g_{21}(x;\kk,q) : = \sgn(\kk) G_{21}(x,x;\kk).
\end{align*}
Given $q\in L^2(\T)$ and $|\kk|\geq \tfrac{1}{2}$, we have that $\g(\kk)\in H^{1}_{|\kk|}(\T)$ with the estimates in \eqref{gbd}, \eqref{g lip}, and \eqref{g lip4}. Bounds on $g_{12}(\kk)$ and $g_{21}(\kk)$ in $H^{1}_{|\kk|}(\T)$ follow then from the identities
\begin{align*}
g_{12}(\kk) = -\frac{q(\g(\kk)+1)}{2\kk-\dd}, \quad g_{21}(\kk) = \frac{\cj{q}(\g(\kk)+1)}{2\kk+\dd}.
\end{align*}
The replacement for \eqref{gammap1} is
\begin{align*}
1+\sgn(\kk)\text{Re} \, \g(\kk) > 0.
\end{align*}
which follows from $\text{Re}\, \jb{\phi ,\mathcal{L}_{\text{NLS}}\phi}=0$ for any complex-valued $\phi\in H^{1}(\R)\times H^{1}(\R)$.

Following \cite{KMV}, we define
\begin{align*}
A(\kk,q)  = \int_{\T} \rho(x;\kk)dx, \quad \text{where} \quad \rho(x;\kk) = \frac{qg_{21}(\kk)-\cj{q}g_{12}(\kk)}{ 2+\g(\kk)},
\end{align*}
and
\begin{align*}
\al(\kk,q) = 2\text{Re}\, A(\kk, q) =  A(\kk,q) - A(-\kk,q)
\end{align*}
for $\kk\geq 1/2$. The quadratic contributions of $\al(\kk,q)$ are
\begin{align*}
\al^{[2]}(\kk,q) = \sum_{\xi \in 2\pi \Z} \frac{ 2\kk}{\xi^2 +4\kk^2}|\ft q(\xi)|^{2},
\end{align*}
so that
\begin{align*}
\al^{[2]}(\kk,q) - \tfrac{1}{2}\al^{[2]}(\tfrac{\kk}{2},q) = \frac{1}{2\kk} \sum_{\xi \in 2\pi \Z} w(\xi,\kk)|\ft q(\xi)|^{2},
\end{align*}
with $w(\xi,\kk)$ as in \eqref{wmult}.
This leads to considering the Gaussian measures $\mu_s$ as in \eqref{gauss0}
and also the candidate invariant measure
\begin{align}
d\wt{\rho}_{s,R} := \wt{Z}_{s,R}^{-1} \ind_{\{ \| q\|_{L^2}^{2} \leq R\}} \exp \bigg( - \int_{1}^{\infty} \kk^{2s} \,\wt{\mathcal{V}}(\kk,q)d\kk \bigg) d\mu_{s}, \label{rhoNLS}
\end{align}
where $\wt{\mathcal{V}}(\kk,q): = \al^{[\geq 4]}(\kk,q) - \tfrac{1}{2} \al^{[\geq 4]}(\kk,q)$, and $\al^{[\geq 4]}(\kk,q):= \al(\kk,q)-\al^{[2]}(\kk,q)$. Regarding the dynamics of \eqref{NLS}, Bourgain~\cite{BO1} proved global well-posedness on $L^2(\T)$, which suffices for our purposes.
Then, we have the following result:

\begin{theorem}[Invariant measures for defocusing cubic NLS]\label{THM:NLS}
	Let $\tfrac{1}{2}<s<1$ and $R>0$. Then, $\wt{\rho}_{s,R}$ in \eqref{rhoNLS} defines a probability measure on $L^2(\T)$, endowed with the Borel sigma algebra, which satisfies: \\
	\textup{(i)} $\wt{\rho}_{s,R}$ is absolutely continuous with respect to the Gaussian measure $\mu_{s}$ in \eqref{gauss0}.\\
	\textup{(ii)} $\supp\wt{\rho}_{s,R} \subseteq H^{\s}(\T)\setminus H^{s-\frac{1}{2}}(\T)$, for every $\s<s-\tfrac{1}{2}$, and moreover,
	\begin{align*}
		\bigcup_{R>0} \supp \wt{\rho}_{s, R} = \supp \mu_{s}.
	\end{align*}
	\textup{(iii)}  The measure $\wt{\rho}_{s,R}$ is invariant under the defocusing cubic NLS flow \eqref{NLS}. More precisely, for $\Phi^{\textup{NLS}}(t):q^0\mapsto q(t)$ the data-to-solution map for defocusing cubic NLS, we have
	\begin{align}
		\int_{L^2(\T)} f( \Phi^{\textup{NLS}}(t)(q)) d\wt{\rho}_{s,R}(q) =\int_{L^2(\T)} f( q) d\wt{\rho}_{s,R}(q)
		\label{NLSvinv}
	\end{align}
	for all $t\in \R$ and for all $f\in L^1(L^2(\T); d\mu_{s})$.
\end{theorem}

Our construction of the measures  \eqref{rhoNLS} follows the same approach as delineated in Section~\ref{SEC:construct}. The invariance of the measures also follows a similar approach as in Section~\ref{SEC:invariance} by considering the approximating Hamiltonian
\begin{align*}
H_{\vk}^{\text{NLS}} := (2\vk)^{2} M - (2\vk)^{3} \al(\vk),
\end{align*}
and the Hamiltonian dynamics it generates
\begin{align}
i\dt q = \frac{\dl H_{\vk}^{\text{NLS}}}{\dl \cj{q}} = 4\vk^3( g_{12}(\vk)-g_{12}(-\vk))+4\vk^2 q. \label{HkNLS}
\end{align}
See \cite[Section 4.1]{HGKV}. The flow \eqref{HkNLS} is globally well-posed in $L^2(\T)$ for any $\vk \geq 1/2$ due to the Lipschitz estimates for $g_{12}(\vk)$ and conservation of the $L^2(\T)$-norm. Moreover, this flow well-approximates the defocusing NLS flow in $L^2(\T)$ in the same sense as in \eqref{mkdvapprox}.
We then define truncated dynamics $H_{\vk,N}^{\text{NLS}}$ similar to \eqref{Hkk-M-diff}.
Again, these truncated dynamics do not leave $\al(\kk)$ invariant so we need an energy estimate analogous to Proposition~\ref{PROP:asym-cons}, which follows along similar lines.
Lastly, the invariance of the Lebesgue measure on $\R^{2N+1}$ under the truncated approximating $H_{\vk,N}^{\text{NLS}}$ dynamics follows as in Lemma~\ref{LM:inv-finite-base} by a careful computation using that the symmetry properties \eqref{symGNLS} imply
\begin{align*}
\frac{\dl g_{12}(x)}{\dl q}(y) =-G_{11}(x,y)^2, &\quad \frac{\dl g_{12}(x)}{\dl \cj{q}}(y) =G_{12}(x,y)^2, \\
\frac{\dl g_{21}(x)}{\dl q}(y) =-G_{21}(x,y)^2, &\quad \frac{\dl g_{21}(x)}{\dl \cj{q}}(y) =G_{22}(x,y)^2.
\end{align*}
The invariance  of $\wt{\rho}_{s,R}$ under the $H_{\vk}^{\text{NLS}}$ and, eventually, the $H^{\text{NLS}}$ flow follows as in the remaining part of Section~\ref{SEC:invariance}.

\begin{ackno}\rm
The authors would like to thank Tadahiro Oh for suggesting the problem and for his continued support. Both authors were supported by the European Research Council
(grant no. 864138 “SingStochDispDyn”). The authors would like to thank the anonymous referee for their helpful comments which have improved the presentation of this manuscript.
\end{ackno}

%\section*{Declarations}
%\begin{comp}\rm
%The authors have no competing interests to declare that are relevant to the content of this article.
%\end{comp}

\end{document}